\documentclass[reqno]{amsart}
\usepackage{hyperref}

\usepackage{amsthm}

\usepackage{enumitem}
\setlist[itemize,1]{leftmargin=\dimexpr 26pt-.065in}
\theoremstyle{definition}
\newtheorem{definition}{Definition}[section]
\newtheorem{theorem}{Theorem}[section]
\newtheorem{proposition}{Proposition}[section]
\newtheorem{remark}{Remark}[section]
\newtheorem{corollary}{Corollary}[section]
\newtheorem{lemma}{Lemma}[section]
\renewenvironment{proof}{{\noindent \bf  Proof.}}{\qed}
\begin{document}
\title[]{On the probabilistic approach to the solution of generalized fractional differential equations of Caputo and Riemann-Liouville type}

\author{M. E. Hern\'andez-Hern\'andez, V. N. Kolokoltsov}

\address{M. E. Hern\'andez-Hern\'andez \newline
Department of Statistics, University of Warwick, Coventry, United Kingdom.}
\email{M.E.Hernandez-Hernandez@warwick.ac.uk}

\address{V. N. Kolokoltsov \newline
Department of Statistics, University of Warwick, Coventry, United Kingdom,
\newline
And Associate member  of Institute of Informatics Problems, FRC CSC RAS.}\email{V.Kolokoltsov@warwick.ac.uk}


\maketitle

\begin{abstract}
This  paper provides a probabilistic approach to solve linear equations  involving  Caputo  and Riemann-Liouville type  derivatives.  Using the  probabilistic interpretation of these operators as the generators of  interrupted  Feller processes, we obtain  well-posedness results and explicit solutions (in terms of the transition densities of the underlying stochastic processes). The problems studied here include   fractional linear differential equations, well analyzed in the literature, as well as their far reaching extensions.
\end{abstract}

\section{Introduction}
The theory of fractional differential equations  is a valuable tool for modeling  a variety of physical phenomena  arising in different fields of science.   Their numerous applications include areas such as engineering, physics, biophysics,  continuum and statistical mechanics, finance, control processing, econophysics, probability and statistics, and so on (see e.g. \cite{carpinteri1997}, \cite{kilbas2}, \cite{podlubny}, \cite{zaslavsky}, and  references therein).  Fractional ordinary differential equations (FODE's) and 
 fractional partial differential equations (FPDE's) have been  used as more accurate models to describe, e.g.,   relaxation phenomena, viscoelastic systems,  anomalous diffusions, L\'evy flights and continuous time random walks (CTRW's)   (see, e.g.,\cite{Bouchaud1990}, \cite{KV}-\cite{KV0}, \cite{FMainardi1997},  \cite{FMainardi2010},   \cite{Meerschaert2012},  \cite{zaslavsky}). 
 
To solve this type of equations different analytical and numerical methods  have been investigated. Analytical methodologies include \textit{the   Laplace,  the Mellin} and  \textit{the Fourier  transform} techniques \cite{kai}, \cite{kilbas2},  \cite{podlubny1994}-\cite{samko},   and the \textit{operational calculus} method \cite{luchko1996}-\cite{hilfer2009}, \cite{gorenflo}.   Regarding  the  numerical approaches, we can mention the \textit{fractional difference method}, the \textit{quadrature formula approach}, the \textit{predictor-corrector approach} as well as some numerical approximations using the \textit{short memory principle} amongst others (see, e.g.,     \cite{kai1997}-\cite{kai}, \cite{edwards}, \cite{podlubny} and references therein).
In the present article we focus on a probabilistic approach to solve  linear differential equations   involving natural extensions (from a probabilistic point of view) of the Caputo and Riemann-Liouville (RL) derivatives of order $\beta \in (0,1)$.

The use of probability theory to solve classical differential equations  is an effective approach   to obtain their solutions  by relating them with boundary value problems of diffusion processes.   In the fractional framework, connections between probability and FPDE's have been also analyzed in the literature \cite{gorenflo98}, \cite{KV}-\cite{KV0}, \cite{Meerschaert2012}, \cite{non1990}, \cite{scalas}. For instance, the probabilistic interpretations of the Green (or fundamental) solution to  the \textit{time-fractional diffusion equation} and \textit{the time-space fractional diffusion equation} are already known   (see references above). 

In this paper we employ similar probabilistic arguments  (transforming the original problem into a Dirichlet type problem)  to study linear equations  involving a general class of Caputo and Riemann-Liouville  type operators. These operators   can be thought of as  the  extensions of the classical  Caputo  and RL fractional derivatives, respectively.   As was shown in \cite{KVFDE}, they can be obtained as the generators of  Markov  processes interrupted on an attempt to cross a boundary point.  The  Caputo  and RL derivatives of order $\beta \in (0,1)$ are particular cases arising by stopping  and killing a $\beta$-stable subordinator, respectively. This fact allows one to solve fractional equations as particular cases of more general equations involving $D-$ and $D_{*}-$operators  of the type $-D_{a+}^{(\nu)}$ and $-D_{a+*}^{(\nu)}$, respectively (see definitions later).    
The problems   we are addressing here are the following:
 \begin{itemize}
 \item [(i)] the   linear equation with the Caputo type operator:
 \begin{equation}\label{case1}
 -D_{a+*}^{(\nu)} u(t) = \lambda u(t)  - g(t),\quad t \in  (a,b],\, \quad u(a)= u_a,
 \end{equation}
  for a given $\lambda \ge 0$, $g$ a bounded function  and $u_a \in \mathbb{R}$.   Since there is a relationship between Caputo and RL type operators (similar to that between the classical Caputo and RL derivatives), we also studied the problem with the RL type operator:
 \begin{equation}\label{case1RL}
 -D_{a+}^{(\nu)} u(t) = \lambda u(t)  - g(t),\quad t \in  (a,b],\, \quad u(a)= 0,
 \end{equation}
 \item  [(ii)] the generalized  \textit{mixed fractional linear equation} 
 \begin{equation}\label{case3}
  \, - \sum _{i=1}^d \, \tilde{D}^{(\nu_i)} \, u(t_1,\ldots, t_d)   =  \lambda u(t_1,\ldots, t_d) - g(t_1,\ldots, t_d),
  \end{equation}
 with some prescribed boundary condition, where $-\tilde{D}^{(\nu_i)}$ denotes either the RL type operator $-\,_{t_i}\!D_{a_i+}^{(\nu_i)}$ or the Caputo type  operator  $-\,_{t_i}\!D_{a_i+*}^{(\nu_i)}$. The  left subscript $t_i$ indicates the operator is acting on the variable $t_i$.
 \end{itemize}

Fractional linear differential equations with Caputo derivatives of order $\beta \in (0,1)$ are particular cases of equation (\ref{case1}). They  have been extensively investigated  by means of the Laplace transform method. Hence, it is  known that \cite{kai}, \cite{podlubny}-\cite{samko}
\begin{equation}\label{linearcase}
  D_{a+*}^{\beta} u(t) =  \lambda u(t)+ g(t), \quad u (a) = u_a, \quad \lambda \in \mathbb{R},
\end{equation}
for $\beta \in (0,1)$ and  a given continuous function $g$ on $[a,b]$,  has the unique  solution 
\begin{equation}\label{General}
u(t) = u_a E_{\beta} \left[\, -\, \lambda (t- a)^{\beta} \right] + \int_a^t  g(r)(t-r)^{\beta-1} E_{\beta,\beta} \left ( \, - \,  \lambda (t-r)^{\beta}\right) dr, \quad t \in (a,b],
\end{equation} 
where $E_{\beta}$ and $E_{\beta,\beta}$ denote the \textit{Mittag-Leffler functions} (see definitions later). 
 This solution can be written in terms of $\beta-$stable densities by means of the integral representation of the Mittag-Leffler functions  given in (\ref{identity}). The probabilistic approach introduced here will give this expression directly once one writes down the expectations involved in the general stochastic representation (\ref{solutionC1B}). On the other hand, using the results  obtained here and the uniqueness of solutions, we obtain a pure probabilistic proof of the  well-known equality  in (\ref{identity}). 

Apart from the classical Caputo derivatives, operators $-D_{a+*}^{(\nu)}$ include, as simple particular cases, the multi-term fractional derivatives $\sum_{i=1}^d \omega_i(t) D_{a_i +*}^{\beta_i} u(t)$ with non-negative functions $\omega_i$. Hence, as another example of (\ref{case1}), our approach also applies to  the \textit{multi-term  fractional equation}   
\begin{equation}\label{multiterm}
\sum_{i=0}^k \omega_i(t)  D^{\beta_i}_{0+*} u(t) =-\lambda u(t) + g(t), \quad  \,\, \beta_i \in (0,1),\,\, t \in \mathbb{R}_+,  
\end{equation}
with some given functions $g$ and $\omega_i,$ $i=1, \ldots, k$.
 The explicit solution for the case when  $\omega_i$   are constants and (\ref{multiterm}) is  a \textit{commensurate equation} (i.e., the quotients $\beta_i/\beta_j$   are rational numbers for all $i,j$),  has  been analyzed  by the reduction of the original problem to either a single- or multi-order fractional differential equation system (see, e.g., \cite{kai}, \cite{edwards}, and references therein).  An approximation for  its solution has  been also studied, e.g., in  \cite{kai}.  
Our approach  encompasses not only the  commensurate case with constant coefficients $\omega_i$  but also  the more general case with non constant coefficients $\omega_i(\cdot)$ and, even more generally,  functions $\beta_i(t)$. 

The fractional counterpart of equation (\ref{case3}) is the \textit{mixed} fractional equation 
 \begin{equation}\label{case3alphabeta}
  -  \, _{t_1}\!D_{0+}^{\beta}   \, u(t_1,t_2) -  \, _{t_2}\!D_{0+*}^{\alpha}   \, u(t_1,t_2) =  \lambda u(t_1, t_2) - g(t_1, t_2), \quad \beta, \alpha \in (0,1),
 \end{equation} 
 subject  to  some boundary condition, where $\lambda \ge 0$ and   $g$ is a given function on $[0,b_1] \times [0,b_2]$. The probabilistic approach presented here provides the explicit solution in terms of $\beta-$ and $\alpha-$stable densities. To our knowledge this  type of mixed fractional equations with Caputo and RL derivatives of order in $(0,1)$ has not been explored explicitly in the literature. The solution to more general cases can be obtained by following an analogous  procedure,  for instance, to solve  the equation
 \begin{equation}\label{case3-Rd}
  -\sum_{i=1}^d  \, \tilde{D}^{\beta_i}   \, u(t_1,\ldots,t_2) =  \lambda u(t_1, \ldots, t_d) - g(t_1, \ldots, t_d),
 \end{equation} 
with $\tilde{D}^{\beta_i}$ being either the RL or the Caputo derivative.
   
The probabilistic ideas behind the solution to (\ref{case1}) and (\ref{case3}) can be used to solve the linear equation with non-constant coefficients, as well as  more general FPDE's involving operators of the type  
$-\!_tD_{a+*}^{(\nu)} - A_{\textbf{x}}$, where  $-\,A_{\textbf{x}}$ denotes the generator of a  Feller process on $\mathbb{R}^d$ acting on the variable $\textbf{x} \in \mathbb{R}^d$.
 Some specific cases include the widely studied  time-fractional diffusion equation (see, e.g.,  \cite{Bouchaud1990}, \cite{kochubei}, \cite{meshart} and references above)
 \begin{equation} \nonumber \frac{\partial^{\beta} u(x,t)}{\partial t^{\beta}} = C \frac{\partial^2 u(x,t)}{\partial x^2}, \quad \,\,\, 0< \beta \le 1, \quad C>0,
\end{equation}
where  $(x,t)$  denotes the space-time variables.
 These cases will be addressed in a forthcoming paper in preparation.
 
It is worth mentioning  that   equations involving $D_{*}-$ and  $D-$operators  do not usually have solutions in the domain of the generators  $-D_{a+*}^{(\nu)}$ and $-D_{a+}^{(\nu)}$, respectively. The existence of such  solutions is restricted to a specific value in the boundary condition. Thus, as usual in classical stochastic analysis, by  introducing the concept of a \textit{generalized solution}  we are able to   study the well-posedness (in a generalized sense) for these equations. To illustrate this, consider   the very-well known   ordinary differential equation (ODE)
\begin{align*}
  u' (t)  = - \lambda u(t) + g (t), \quad t \in  (0,b],  \quad 
  u(0) = u_0, 
  \end{align*}
where $b>0$, $g \in C[0,b]$ and $\lambda  > 0$, and whose solution is 
\begin{equation}\label{solLinearD}
  u (t)  =  u_0 e^{\lambda t} + \int_0^t exp\{\lambda (t-r)\} g(r)dr, \quad t \in (0,b].\end{equation}
Probabilistically, this problem can be thought of as the boundary value problem associated with the  deterministic linear motion on $(-\infty,b]$ which is stopped at reaching the boundary point $t= 0$. In this case, the semigroup $\{S_s\}_{s \ge 0}$  of the deterministic  process is given by $S_s f(t) = f(t - s)$ for any $t  \in (-\infty,b]$ and $f \in C(-\infty,b]$, whilst   the semigroup  $\{S_s^{0+*}\}_{s\ge 0}$ for the stopped process is $S_s^{0+*} f(t) = f(\max\{0,t - s\})$  for any $t \in [0,b]$ and $f \in C[0,b]$. Hence, the resolvent operator of  the semigroup $S_s^{0+*}$ provides   equation (\ref{solLinearD}) as the unique   \textit{solution in the domain of the generator} (the space $C^1[0,b]$) if, and only if,  $u_0 = \frac{1}{\lambda }g(0)$.     Otherwise  this solution is (in our terminology) only a \textit{generalized solution} as it can be obtained as a limit of solutions in the domain of the generator. Moreover, in this case the  generalized solution is also  a  \textit{classical (smooth) solution} lying  on $C[0,b] \cap C^1(0,b]$ instead of   $C^1[0,b]$.  Similar situations  occur when considering   fractional differential equations. We will see that the solutions found in the literature are usually solutions in the generalized sense, as they usually do not belong to the domain of $-D_{a+*}^{\beta}$ or $-D_{a+}^{\beta}$ as generators of Feller processes. 

The main contribution of this work relies on providing well-posedness results and explicit integral representations of solutions to linear equations with Caputo and RL type operators. We pay attention to the existence  of two  types of solutions: \textit{solutions in the domain of the generator} 
and  \textit{generalized solutions}. The latter concept  defined for rather general, even not continuous, functions $g$ in (\ref{case1}).  Moreover, all solutions  are given  in terms of expectations of functionals of Markov processes.     From the point of view of numerical analysis, this representation can be exploited  to obtain numerical solutions to a variety of problems by performing Monte Carlo techniques.  Simulation methods have been effectively used for classical differential equations and, in recent years,   different methods for evaluating path functionals of L\'evy processes have been actively researched (see, e.g., \cite{Castilla2012}-\cite{Castilla2014}, \cite{Kuznetsov2011}). 
Finally, by using the monotonicity of the underlying processes,  we obtain   explicit solutions   in  terms of the transition densities of the Markov processes  involved. 

The paper is organized as follows. The next section provides  some standard notation and gives a quick summary about Caputo and RL  derivatives as well as $\beta-$stable subordinators.   Section 3  introduces the definition of the generalized  RL and  Caputo type operators,  $-D_{a+}^{(\nu)}$ and $ -D_{a+*}^{(\nu)}$, respectively. In Section 4 we study important properties of the underlying stochastic processes associated with the generalized fractional operators.    Then, the probabilistic solution to the equation   (\ref{case1})  is addressed in Section 5. This section concludes with the application of these results  to deduce the known solutions to  fractional linear equations with the classical Caputo derivatives  $D_{a+*}^{\beta}$, for $\beta \in (0,1)$.  The last section presents the well-posedness  to the mixed  equations (\ref{case3}) and (\ref{case3alphabeta}).
 
\section{Preliminaries}
\subsection{Notation}
 Let $\mathbb{N}$, $\mathbb{C}$,   and $\mathbb{R}^d$  be the set of positive integers, the complex space    and  the $d$-dimensional Euclidean space, respectively.  For any interval $A$, the standard notation  $B(A)$, $C(A)$,  $C_b(A)$ and $C^1(A)$ will denote the set of bounded Borel measurable functions, continuous functions,  bounded continuous functions and continuously differentiable functions defined on $A$, respectively.   If $A$ is a  closed set, $C^1(A)$ means the space of continuous functions differentiable up to the boundary. Notation $||\cdot|| $ stands for the sup-norm $||f|| = \sup_{x \in A} |f(x)|$ for $f \in B(A)$.  If  $A = [a,b]$, notation $C_a[a,b]$ and $C_a^1[a,b]$  denote  the space of continuos functions vanishing at $a$ and the space $C_a[a,b] \cap C^1[a,b]$, respectively.
 
As usual, $E_{\beta}$ will refer to the \textit{Mittag-Leffler function} of order $\beta > 0$ defined by
\begin{equation}\nonumber E_{\beta} (z) := \sum_{j=0}^{\infty} \frac{z^j}{\Gamma (j\beta+ 1)}, \quad z \in \mathbb{C},\, 
\end{equation}
whilst  $E_{\beta_1, \beta_2}$ means the \textit{two-parameter Mittag-Leffler function}:
\begin{equation}\nonumber E_{\beta_1, \beta_2} (z):= \sum_{j=0}^{\infty} \frac{z^j}{ \Gamma ( j\beta_1 + \beta_2)}, \quad z \in \mathbb{C},\, \,\, \beta_1, \beta_2 > 0.
\end{equation}
For properties of these functions see, e.g.,  \cite{kai}, \cite{podlubny}-\cite{samko}.

Letters $t$ and $r$ are mainly used as space variables, and the letter $s$ is used as a time variable. Bold letters, e.g. $\textbf{t}$ and $\textbf{a}$, shall  denote elements in $\mathbb{R}^d$ for $d\ge 2$, and  bold capital letters will stand for $\mathbb{R}^d-$valued stochastic processes, e.g. $\mathbf{T}=\{ \mathbf{T}(s)\,:\, s \ge 0\}$. Letters $\mathbf{P}$ and $\mathbf{E}$  are reserved   for the probability and the mathematical expectation, respectively.   Notation $W_{\beta}(\sigma, \gamma)$ means a $\beta-$stable random variable (r.v.) with scaling parameter $\sigma$, skewness parameter $\gamma$ and location parameter zero. Its  density function  shall be denoted by   $w_{\beta}(\cdot; \sigma, \gamma)$. Finally, for a given Feller semigroup $\{S_t\}_{t \ge 0}$ on $C_b(S)$, its \textit{resolvent operator} at $\lambda > 0$ is denoted by $R_{\lambda}$ and  is defined as the Bochner integral (see e.g. \cite{EK})  
\begin{equation}\label{resolvent}
R_{\lambda} g\,:= \int_0^{\infty} e^{-\lambda s} S_s g\,ds, \quad g \in C_b(S).
\end{equation}
By taking $\lambda = 0$ in (\ref{resolvent}), one  obtains  the \textit{potential operator} which is denoted by $R_0g$ (whenever it exists). Additional superscripts shall be used to identify different resolvents and potential operators.
 
 \subsection{Fractional differential operators}

This section provides a quick summary of basic results concerning the classical Riemann-Liouville and Caputo fractional operators. For a detailed treatment refer, e.g.,  to   \cite{podlubny}-\cite{samko}  and references therein.

 Classical \textit{fractional differential operators} are defined in terms of both the standard differential operator (hereafter denoted by $D^m$, $m \in \mathbb{N}$),  and the  \textit{Riemann-Liouville integral fractional operator} $I^{\alpha}_{a+}$ for any $\alpha > 0$, $a \in \mathbb{R}\cup \{-\infty\}$,    defined by
\begin{equation}\nonumber 
  I_{a+}^{\alpha} h(t)  := \frac{1}{ \Gamma (\alpha)} \int_a^t (t-s)^{\alpha -1} h(s)ds, \quad \alpha >0,\,t > a.
  \end{equation}
For convention, $I_{0+}^{0}$ refers  to the identity operator.  

 The \textit{left-sided    Riemann-Liouville   ($RL$)  operator   of order  $\beta>0$} (shortly the RL  derivative)  is defined   as the left-inverse of the corresponding integral operator  (\cite{kai}, \cite{samko}). To be precise,   
 if   $\beta \in \mathbb{R}^+$  and $m = \lceil \beta \rceil$ ($\, \lceil \cdot \rceil$ denoting the ceiling function), then the $RL$  derivative $D_{a+}^{\beta}$ is defined by
\begin{equation}\label{RLD}
D_{a+}^{\beta} h (t) : = D^{m} I_{a+}^{m - \beta} h (t), \quad  \beta > 0, \,\, \beta \notin \mathbb{N}, \,\, t>a. 
\end{equation}
A sufficient condition for $D_{a+}^{\beta}$ to be well-defined  is to assume that  $f \in A^m[a, \infty)$, i.e., its derivatives of order $m-1$ are absolutely continuous (see, e.g., \cite{kai}). 

An alternative fractional differential operator  is the  \textit{left-sided Caputo   operator} (shortly the Caputo derivative): 
\begin{equation}\label{caputo}
D_{a+*}^{\beta} h (t)   := I^{m- \beta}_{a+} D^m h(t), \quad \beta > 0, \,\, \beta \notin \mathbb{N}, \quad t > a,
\end{equation}
for which $h$ requires the absolute integrability of its derivatives of order $m= \lceil \beta \rceil$. 

It can be proved (see e.g. \cite{kai}) that both operators are related by the equality 
\begin{equation}\label{otherCaputo}
D_{a+*}^{\beta} h (t)   = D_{a+}^{\beta}  [ h - T_{m-1} [h;a] ],
\end{equation}
where $T_{m-1}[h;a]$ denotes the Taylor expansion of  order $m-1$, centered at $a$, for the function $h$. Hence,  in general
\[   D^{\beta}_{a+}h (t) := D^m I^{m- \beta}_{a+} h(t)  \,\,\, \neq \,\,\, I^{m-\beta} D^m h(t)=: D_{a+*}^{\beta}h(t), \] 
unless the function $h(t)$ along with its first $m-1$ derivatives vanish at $a+$ (or  as $t \to -\infty$ for the case $a = - \infty$).

\begin{remark}
The left-sided derivatives have a direct counterpart to the right-sided versions (see previous references for details). 
However, we work everywhere in this paper only with the left-sided operators and their generalizations. The right-sided version of these results is a straightforward modification.
\end{remark}

\subsection{Special case: $\beta \in (0,1)$}

To solve fractional differential equations, we will be  mostly  interested  in fractional derivatives of order $\beta \in (0,1)$. In this case  equations (\ref{RLD}) and (\ref{caputo}) become  
\begin{align*}
  D_{a+}^{\beta}h (t) &= \frac{1}{\Gamma ( \beta)}   \frac{d}{dt}     \left (\int_a^t (t- r)^{-\beta } h(r)dr \right ),  \quad t > a,
  \end{align*}
and
\begin{align*}
  D_{a+*}^{\beta}h (t) &= \frac{1}{\Gamma ( \beta)}  \int_a^t (t- r)^{-\beta }       h'(r)dr ,  \quad t > a,
  \end{align*}
respectively. 
Further,  for smooth enough functions $h$ (e.g.   $h$ in the Schwartz space), one obtains
\begin{align}\label{RL}
  D_{a+}^{\beta}h (t) &= \frac{1}{\Gamma (- \beta)} \int_0^{t-a} \frac{h(t-r) - h(t) }{r^{1+\beta}} dr + \frac{h(t)}{ \Gamma (1-\beta) (t-a)^{\beta}}, \quad t> a, \,\,\beta \in (0,1),
  \end{align}
and
\begin{align}\label{Caputo}
  D_{a+*}^{\beta}h (t) &= \frac{1}{\Gamma (- \beta)} \int_0^{t-a} \frac{h(t-r) - h(t) }{r^{1+\beta}} dr + \frac{h(t) - h(a)}{ \Gamma (1-\beta) (t-a)^{\beta}}, \quad t> a, \,\, \beta \in (0,1),
  \end{align}
    for details see, e.g., Appendix in \cite{KVFDE}. 
 
Thus, for these values of $\beta$  the relationship between the Caputo and the RL derivates in  (\ref{otherCaputo})    translates to
\begin{align*}
 D_{a+*}^{\beta}h(t)  &=  D_{a+}^{\beta} [\, h - h(a) ] (t)  
 =  D_{a+}^{\beta} h(t) -  \frac{h(a)}{ \Gamma (1 - \beta) (t-a)^{\beta}},\quad t > a.
\end{align*} 
For smooth bounded integrable functions or  functions that vanish at $t = a$ (or as $ t \to -\infty$ for the case $a = - \infty$), the previous equality implies that the Caputo  derivative and the  RL  derivative coincide. Its common value for $a = -\infty$, from now on denoted by $d^{\beta}/dt^{\beta}$, is sometimes called    \textit{the generator form of the fractional  derivative of order $\beta \in (0,1)$}, \cite{Meerschaert2012}: 
\begin{equation}\label{gen}
\frac{d^{\beta}}{dt^{\beta}} h (t) := D_{-\infty +}^{\beta}h (t)  = D_{-\infty +*}^{\beta}h (t)  =   \frac{1}{\Gamma (-\beta)} \int_0^{\infty} \frac{h(t-r) -h(t)}{r^{1 + \beta}}dr. 
\end{equation}
\subsection{Stable subordinators}
Hereafter, we will always assume the existence of a probability space $(\Omega, \mathcal{G}, \mathbb{P})$ such that  all the stochastic processes of our interest are defined on it.   Notation $\mathcal{F}_s^{X}$ means the completed natural filtration generated by a process $X=\{X(s)\}_{s\ge 0}$, i.e. $\mathcal{F}_s^{X}:= \sigma (X_r\,:\, 0\le r \le s)$.

A $\beta$-stable subordinator for $\beta \in (0,1)$, is a  real-valued stable L\'evy process  $T^{\beta} = \{T^{\beta}(s)\}_{s \ge 0}$ started at $0$ almost surely (a.s.) with  independent increments, $T^{\beta}(s) - T^{\beta}(r)$ for any $ 0\le r < s$,     having the same distribution as the r.v. $ W_{\beta} ((s-r)^{1/\beta},1)$,  a totally skewed  positive $\beta$-stable r.v. with scale parameter $\sigma = (s-r)^{1/\beta}$,  see e.g. \cite{a}, \cite{taqqu}.  

This process has nondecreasing sample paths a.s., and it has only positive jumps. Moreover,  this is time-homogeneous with respect to its natural filtration. Further, since $\beta$-stable processes are self-similar  with index $1/\beta$,   the process $\{c^{1/\beta} T^{\beta}(s)\}_{s \ge 0}$ has the same distribution as  the process $\{T^{\beta}(cs)\}_{s \ge 0}$ for any positive constant $c$. 
 Consequently,   the transition probabilities $p_s^{\beta}(t,E):= \mathbf{P} [T^{\beta}(s) \in E | T^{\beta}(0) = t]$ for any $E\in \mathcal{B}(\mathbb{R})$ (the Borel sets of $\mathbb{R}$) satisfy
\begin{equation} \nonumber 
p_s^{\beta}(t,E) = s^{-1/\beta} \int_E  w_{\beta} (s^{-1/\beta} (r-t); 1,1)dr,
\end{equation}
where $w_{\beta}(\cdot;1,1)$  is the   density of a standard  $\beta-$stable r.v. $W_{\beta}(1,1)$. The density in this case  equals   
\begin{align*}
w_{\beta}(x;1,1)  &=\frac{1}{\pi} \Re \int_0^{\infty} \exp \left \{ -iux -u^{\beta}\exp \left( -i \frac{\pi}{2} \beta \right )\right \} du,
\end{align*}
where $\Re (z)$ means the real part of $z \in \mathbb{C}$ (see Theorem 2.2.1 in \cite{zolotarev}).

The  infinitesimal generator of a $\beta$-stable subordinator,   denoted by $A^{\beta}$,  is the generator of a jump-type Markov process of the form
\begin{equation}\label{generatorL}
  A^{\beta} h (t)  = \int_0^{\infty} (h (t+ r) - h(t)) \nu_{\beta} (dr), \quad h \in \mathfrak{D}_{\beta},  
  \end{equation}
with a domain $\mathfrak{D}_{\beta} $ and with the jump  intensity given by the L\'evy measure $\nu$ supported in $\mathbb{R}_+$: 
 \begin{equation}\label{levym}
\nu_{\beta}(dr)= \frac{\beta}{ \Gamma (1- \beta)r^{1+ \beta}}dr =  \,-\,\frac{1}{ \Gamma (- \beta) r^{1+ \beta}}dr. 
\end{equation} 
The last equality holds due to the identity $\Gamma (x) =  (x-1) \Gamma (x-1)$. 

For $\beta \in (0,1)$, we will say that the process $T^{+\beta} = \left\{ T^{+\beta} (s)\,:\, s \ge 0 \right \}$  is an \textit{inverted $\beta-$stable subordinator} if $-T^{+\beta}$ is a $\beta-$stable subordinator.   Thus, $T^{+\beta}$ is a  Markov process with non increasing sample paths a.s.  and with the generator
\begin{equation} \nonumber
A^{+\beta} h(t) = \int_0^{\infty} \left( h(t-r) -h(t) \right ) \nu_{\beta}(dr).
\end{equation}
Notice that  the relation
\begin{equation}\nonumber
w_{\beta} (-t; \sigma, 1) = w_{\beta} (t; \sigma, -1),
\end{equation}
implies that $T^{+\beta} (s) - T^{+\beta}(r)$ has the same distribution as  the r.v. $ W_{\beta} \left( (s-r)^{1/\beta}, -1,0\right)$. Hence,  the transition probabilities $p_s^{+\beta}(t,E):= \mathbf{P} [T^{+\beta}(s) \in E | T^{+\beta}(0) = t]$ are given by \begin{equation}\label{transitionpXInv}
p_s^{+\beta}(t,E) = s^{-1/\beta} \int_E  w_{\beta} (s^{-1/\beta} (t-r); 1,1)dr,\quad E \in \mathcal{B}(\mathbb{R}).
\end{equation}
\section{Generalized  fractional operators}\label{definitionD}

Let $T^{\beta}=\{T^{\beta}(s):s \ge 0\}$ be a  $\beta-$stable subordinator  with the generator $A^{\beta}$ and let  $T^{+\beta} =\{T^{+\beta}(s):s \ge 0\}$ be the corresponding inverted $\beta-$stable subordinator  with the generator   $A^{+\beta}$.  Then  the operator $-d^{\beta}/dt^{\beta}$ in (\ref{gen}) coincides with $A^{+\beta}$, i.e.  
\begin{align*}
 - \frac{d^{\beta}}{dt^{\beta}}h(t)  = D_{-\infty +}^{\beta} h (t)  = D_{-\infty +*}^{\beta} h (t)  =  A^{+\beta} h (t). 
\end{align*}
As was shown in \cite{KVFDE}, an analogous probabilistic interpretation of  the Caputo operators $D_{a+*}^{\beta}$ (resp.  $D_{a+}^{\beta}$),  for any  $a \in \mathbb{R}$ and $\beta \in (0,1)$,   can be obtained by   interrupting (resp. killing )  $T^{+\beta}$   on an attempt to cross the boundary point $t=a$. Moreover, this interruption procedure  naturally yields  an extension  of the Caputo and the RL fractional derivatives.

  Namely, let $-D_+^{(\nu)}$ be the  generator of a decreasing Feller process taking values on $(-\infty,b]$, $b\in \mathbb{R}$,  given by
  \begin{equation}\label{Adecreasing}
  -D_+^{(\nu)}h(t) = \int_0^{\infty} \left (h(t-r) - h(r)  \right ) \nu(t,r)dr, \quad t \le b,
  \end{equation}
  with a  function $\nu(t, r)$ satisfying the condition:  
  \begin{itemize}
  \item [(H0)] the function $\nu(t, r)$ is  continuous as a function of two variables and continuously  differentiable in the first variable.    Furthermore,
  \begin{equation}\nonumber 
    \sup_t \int r \nu (t,r)dr < \infty, \quad \sup_t \int r \Big | \frac{\partial}{\partial t} \nu(t,r)\Big |dr < \infty,
  \end{equation}
  and 
  \begin{equation}\nonumber   \lim_{\delta \to 0} \sup_t \int_{|r| \le \delta} r \nu(t,r)dr = 0. 
   \end{equation}
  \end{itemize}
  Then,  by Theorem 4.1 in \cite{KVFDE},  the  process interrupted (and forced to land exactly at $t=a$)  on the first attempt to cross the barrier  point $t=a$  is a Feller process on $[a,b]$ and  has the  generator
  \begin{align}
  -D_{a+*}^{(\nu)} h(t)   &= \int_0^{t-a} ( h(t-r) - h(t)) \nu(t,r)dr + (h(a) - h(t)) \int_{t-a}^{\infty}\nu(t,r)dr, \quad t \in [a,b]\label{genaP}
  \end{align}
  with the invariant core $C^1 [a, b]$ and a domain denoted by $\mathfrak{D}_{a+*}^{(\nu)}$. Note that, since  the process is  decreasing, the interruption procedure   effectively means  stopping the process at the boundary.
   
Moreover, if the  process is also killed  at the barrier point $t=a$ (meaning analytically to set $h(a) = 0$) then the corresponding  Feller sub-Markov process on $(a,b]$   has the  generator
  \begin{align}
  -D_{a+}^{(\nu)} h(t)  &:=  \int_0^{t-a} ( h(t-r) - h(t)) \nu(t,r)dr -  h(t) \int_{t-a}^{\infty}\nu(t,r)dr,\quad t \in (a,b],\label{genaP2}
  \end{align}
  with the invariant core $C_a^1[a,b]$ and a domain $\mathfrak{D}_{a+}^{(\nu)}$.
\begin{definition}
  Let  $\nu$ be a function satisfying condition (H0).   The  operator in (\ref{genaP})   will be called   \textit{the  $D_{*}$-operator   defined by $\nu$} and will be denoted by $- D_{a+*}^{(\nu)}$ (the sign $-$  is introduced to comply with the standard notation of fractional derivatives). Similarly, the operator in (\ref{genaP2}) will be called \textit{the  $D$-operator   defined by $\nu$} and will be denoted by $- D_{a+}^{(\nu)}$.
 \end{definition}
 \begin{remark}
Since we are interested in solutions to equations on finite intervals $[a,b]$, we have applied the results in \cite{KVFDE} to the case of stochastic processes taking values on  $(-\infty,b]$ for some $b \in \mathbb{R}$. 
 \end{remark}
By the standard theory of Feller processes, the domain of the generators  $-D_{a+*}^{(\nu)}$ and $-D_{a+}^{(\nu)}$ coincides with the images of their corresponding   \textit{resolvent operators},
 $R_{\lambda}^{a+*(\nu)}$ and $R_{\lambda}^{a+(\nu)}$ (for any $\lambda > 0$), respectively.  Namely, 
\begin{equation}\nonumber
\mathfrak{D}_{a+*}^{(\nu)} = \left \{ u_g\,\Big |\ u_g(t) := R_{\lambda}^{a+*(\nu)} g(t);\,\, \,g \in C[a, b]  \right \}, 
\end{equation}
and
\begin{equation}\nonumber 
\mathfrak{D}_{a+}^{(\nu)} = \left \{ u_g \,\Big |\ u_g(t) := R_{\lambda}^{a+(\nu)} g(t);\,\, g \in C_a[a, b] \right \}. 
\end{equation}
Moreover, the images of the resolvent operators are independent of $\lambda$ (for details see, e.g.,  \cite{dynkin1965}).
The operators in  (\ref{genaP})   and  (\ref{genaP2}) can be thought of as the generalization of the Caputo and the Riemann-Liouville operators of order $\beta \in (0,1)$, respectively. 

Particular cases of the $D_{*}-$ operators include   the \textit{multi-term  fractional operators}:\begin{equation}\nonumber
-\,D_{a+*}^{(\nu)} h(t) := \,-\, \sum_{i=1}^d \omega_i (t)D_{a+*}^{\beta_i} h(t),   
\end{equation}
with non-negative  functions $\omega_i(\cdot) \ge 0$ and\[ \nu (t,r) = \,-\, \sum_{i=1}^d \omega_i(t) \frac{1}{\Gamma(-\beta_i) r^{1+\beta_i}}.\] 
Even more generally,   they include the case
\begin{equation}\label{Gdistributed}
-\,D_{a+*}^{(\nu)} h(t) :=\, -\, \int_{-\infty}^{\infty} \omega(s,t) D_{a+*}^{\beta(s,t)} h(t)\, \mu (ds),   
\end{equation}
 with  \[ \nu(t,r)= \,-\,\int_{-\infty}^{\infty} \omega (s,t) \frac{ds}{\Gamma(-\beta(s,t))r^{1+\beta(s,t)}}, \] satisfying condition ($H0$).  Particular cases of (\ref{Gdistributed}) have been studied e.g. in \cite{distributed1, distributed2}. 

\subsection{Classical fractional derivatives} 
 
For any  $\beta \in (0,1)$,  the  fractional operators  $D_{a+*}^{\beta}$ and $D_{a+}^{\beta}$  are obtained as particular cases of $D_{*}$- and $D$-operators, respectively. Namely,  on smooth enough functions $h$,   
\[ \text{if }\quad  \small{ \nu(t,r)=\, -\, \frac{1}{\Gamma(-\beta) r^{1+\beta}} }, \normalsize  \quad \beta \in (0,1) \quad   \text{ then }\quad 
\begin{array}{lcl} 
-D_{a+*}^{(\nu)} h (t)  &= & \,-\,D_{a+*}^{\beta} h(t), \\
-D_{a+}^{(\nu)} h (t)  &= & \,-\,D_{a+}^{\beta} h(t).      
\end{array}
\]
Thus, $-D_{a+*}^{\beta}$ is the generator of a Feller process on $[a, b]$ which is obtained by stopping an inverted  $\beta$-stable subordinator $T^{+\beta}$  in the attempt on crossing the boundary point $t=a$.  The RL derivative $-D_{a+}^{\beta}$ is the generator obtained by killing  $T^{+\beta}$ at the boundary,  see  \cite{KVFDE}.
 \begin{remark}
 The probabilistic interpretation of fractional derivatives of order $\beta \in (1,2)$ was also analyzed  in \cite{KVFDE}  but this case will not be needed  here.
 \end{remark}
 
 \subsection{Fractional derivatives of position-dependent  order} 
 For a given function $\beta: \mathbb{R} \to (0,1)$,  define 
 \begin{equation}\label{betax}
 \nu(t,r) =  \,- \, \frac{1}{\Gamma(-\beta(t))  r^{1+ \beta(t)} }.
 \end{equation}
 \begin{lemma}\label{T41bx}
  If $\beta : \mathbb{R} \to (0,1)$ is a continuously differentiable  function with values in a compact subset of $(0,1)$, then the function  defined  in (\ref{betax}) satisfies condition (H0).
    \end{lemma}
    \begin{proof} Follows by the smoothness of the function $\beta$ in a compact set of $(0,1)$.  \end{proof} 
    Lemma  \ref{T41bx}   allows us to define  $D_{*}-$operators by using   (\ref{betax}). They are  denoted by $-D_{a+*}^{(\nu)}  \equiv \,- \, D_{a+*}^{\beta(t)} $ and can be seen as    generators of  inverted   \textit{stable-like processes} \cite{KV0} with jump density (\ref{betax}) which are stopped at the boundary point $t = a$. These operators will  be  referred as \textit{Caputo-type operators of position dependent order}.  Analogously, the RL type operators defined by (\ref{betax}) will be denoted by $-D_{a+}^{(\nu)} \equiv -D_{a+}^{\beta(t)}$.

 \begin{remark}
 Note that in the previous case the 'order' of the derivative depends  only on $t$.  Some other extensions can be made by taking the  function $\nu$ depending on external variables. 
This case (which we analyze  in detail in a forthcoming paper)   allows us to deal with   operators  of the form 
 \begin{equation}\nonumber 
 \left ( \,-\, _{t}\!D_{ \,a+*}^{\beta(t,\textbf{x})}  - A_{\textbf{x}}^{(t)}  \, \right ) h(t,\textbf{x}), \quad t \ge a, \,\, \textbf{x} \in \mathbb{R}^d,
 \end{equation}
 where $-\,_{t}\!D_{a+*}^{\beta(t,\textbf{x})} $ denotes the Caputo type  derivative acting on the variable $t$ and depending on the variable  $\textbf{x}$ as a parameter; and   $-\,A_{\textbf{x}}^{(t)}$ denotes the generator of a  Feller process acting on the variable $\textbf{x}$ and depending on the variable $t$ as a parameter.
\end{remark}
  
\section{Properties of the underlying stochastic processes}
In this section we study some facts about the underlying stochastic processes generated by the operators $-D_{a+*}^{(\nu)}$ and $-D_{a+}^{(\nu)}$. These results shall be used to obtain the explicit  solutions to the linear equations involving $D_{*}-$ and $D-$operators. 

For a given  function $\nu$  satisfying condition (H0) and for $t \in (a, b]$, the notation  \begin{equation}\nonumber 
  \, T^{+(\nu)}_t , \,-D^{+(\nu)},  \, T^{a+*(\nu)}_t, \, -\,D_{a+*}^{(\nu)},\,\,  \,T_{t}^{a+(\nu)} \,\,\,\text{ and } -D_{a+}^{(\nu)},
  \end{equation}
means the following:  $T^{+(\nu)}_t = \{T^{+(\nu)}_t (s)\,:\,s\ge 0\}$  is the  decreasing Feller process (started at $t$)   generated by  $-D^{+(\nu)}$ as given in  (\ref{Adecreasing});  $T_{t}^{a+*(\nu)} = \{T^{a+*(\nu)}_t (s)\,:\,s\ge 0\}$ stands for  the Feller process generated by $ -\,D_{a+*}^{(\nu)}$ with the invariant core   $C^1[a, b]$,  which arises by  interrupting $T_t^{+(\nu)}$ on an attempt to cross the boundary point $a$; and  $T_{t}^{a+(\nu)}= \{T^{a+(\nu)}_t (s)\,:\,s\ge 0\}$ denotes  the  Feller sub-Markov process generated by $ -D_{a+}^{(\nu)}$ with the invariant core   $C_a[a,b]$,  which arises by   killing $T_t^{a+*(\nu)}$ at the point $a$.

For $t \in [a,b]$, notation  $\tau_a^{t,(\nu)}$ refers to  the first time the process $T_{t}^{+(\nu)}$ (or the process $T_{t}^{a+*(\nu)}$) leaves $(a,b]$, i.e.
 \begin{equation} \nonumber   \tau_a^{t,(\nu)}  :=  \inf \left \{ s \ge 0\,:\, T_{t}^{+(\nu)} (s) \notin (a,b]\, \right \} = \inf \left \{ s \ge 0\,:\, T_{t}^{a+*(\nu)} (s) \notin (a,b]\,\right \},
  \end{equation}
and, of course, $\tau_a^{a,(\nu)}= 0$. Note that $ \tau_a^{t,(\nu)}$  is a stopping time with respect to    $\mathcal{F}_s^{T_{t}^{+(\nu)}}$. 
 Further,  $\tau_a^{t,(\nu)}$ is also the first exit time from $(a,b]$ of the killed process $T_t^{a+(\nu)}$. 
 
Let $p_s^{+(\nu)}(t,E):= \mathbf{P} \left[ T_t^{+(\nu)}(s) \in E \,|\, T_t^{+(\nu)} (0) = t \right ]$ be the transition probabilities of $T_t^{+(\nu)}$ from $t$ to $E \in \mathcal{B}(-\infty,b]$ during the interval $[0,s]$. Note that if  $p_s^{a+*(\nu)}(t,E)$   denotes the corresponding transition probabilities of $T_t^{a+*(\nu)}$, then
\begin{equation}\nonumber p_s^{a+*(\nu)}(r,E) = \left \{ 
\begin{array}{ll}
 p_s^{+(\nu)}(r,E),& E \in \mathcal{B}(a, b] \\
p_s^{+(\nu)} (r, (-\infty,a]), & E = \{a\}.
 \end{array}
 \right . \quad r \in (a,b].
\end{equation}
Since the process is stopped at $a$, 
\begin{align*}
p_s^{a+*(\nu)} (t,[a,t]) &=  p_s^{+(\nu)} (t,(a,t]) + p_s^{a+*(\nu)} (t,\{a\}) = 1,
\end{align*}
and   $\,p_s^{a+*(\nu)}(a,E) = 1$ whenever  $E \in \mathcal{B}[a,b]$ and $a \in E$.  Furthermore, if $p_s^{a+(\nu)}(r,E)$ denotes the transition probabilities of  $T_t^{a+(\nu)}$, then  $p_s^{a+(\nu)}(r,E) = p_s^{+(\nu)} (r, E)$ for all $r \in (a,t]$ and $E \in \mathcal{B} (a,b]$. Moreover, 
\begin{align*}
p_s^{a+(\nu)} (t,(a,t]) &=  p_s^{+(\nu)} (t,(a,t])  = 1 -  p_s^{+(\nu)} (t,(-\infty,a])  \le  1.
\end{align*}
The previous implies that 
\[ p_s^{+(\nu)} (r,E) = p_s^{a+*(\nu)} (r,E) = p_s^{a+(\nu)} (r,E), \quad r \in(a,t],\]
on Borel sets $E$ of $(a,b]$.

  Sometimes we will use the following  additional assumptions   concerning the   function $\nu$  and the transition  probabilities  of the underlying process $T^{+(\nu)}$:
\begin{itemize}
\item [(H1)]  There exist $\epsilon >0$ and $\delta > 0$  such that  $ \nu (t,r) \ge \delta > 0$ for all $t$ and all $|r| < \epsilon$.
\item [(H2)]  The  transition probabilities of the process  $T^{+(\nu)}$ are absolutely continuous with respect to the Lebesgue measure (the transition densities will be denoted by $p_s^{+(\nu)}(r,y)$). \item [(H3)] The transition density function $p_{s}^{+(\nu)}(r,y)$ is continuously differentiable in the variable $s$.
\end{itemize}
\begin{remark}
There exists several criteria in terms of $\nu$ that ensure the validity of ($H2$) and ($H3$), see, e.g.,  \cite{H}.
\end{remark}

\begin{lemma}\label{propH0}
Suppose the conditions (H0)-(H1)  hold for a function $\nu$.  
Then,  the stopping time $\tau_a^{t,(\nu)}$ is finite a.s. and $\mathbf{E}\left [\tau_a^{t,(\nu)} \right ] < + \infty$ uniformly on $t\in (a,b]$. Further, the point $a$ is  regular in expectation  for both  operators $-D_{a+*}^{(\nu)}$ and $-D_{a+}^{(\nu)}$, i.e.  $\mathbf{E}\left [\tau_a^{t,(\nu)} \right ] \to 0$ as $ t \downarrow a$.
\end{lemma}
\begin{proof}
 The result   follows  by comparing the process $T_t^{+(\nu)}$ with a compound Poisson process with L\'evy kernel $\nu(r) = \delta \mathbf{1}_{\{[0,\epsilon]\}}(r)$ for which the result holds. 
\end{proof}
\begin{remark}
It is worth stressing that the assumption (H0) is weaker to the one used in Theorem 4.1 in \cite{KVFDE}.
\end{remark}
Observe now that, since $T_{t}^{a+*(\nu)}$ is a decreasing process, the equivalence between the  events $\left \{\, \tau_a^{t,(\nu)} > s\,\right \}$ and  $\left \{\,T_{t}^{a+*(\nu)} (s) > a\,\right \}$   for $t \in (a,b]$ and all $s > 0$ implies
    \begin{align}
  \mathbf{P}\left [\tau_a^{t,(\nu)}  > s\right ]  &=  \mathbf{P} \left [ T_{t}^{a+*(\nu)} (s) > a \right ]  =   \int_{(a,t]} p_s^{a+*(\nu)} (t,r)dr = 1- \int_{-\infty}^a p_s^{+(\nu)} (t,r)dr, \nonumber  
   \end{align}
     yielding the following result.
  \begin{proposition}\label{muTau}
   Suppose the conditions (H0)-(H3) hold. Then, the probability law of $\tau_a^{t,(\nu)} $, denoted by $\mu_a^{t, (\nu)} (ds)$,   is absolutely continuous  with respect to  Lebesgue measure  for $t \in(a,b]$ and its density $\mu_a^{t,(\nu)}(s)$ is  given by
    \begin{equation}\label{lawofTau}
\mu_a^{t,(\nu)} (s)  =   \frac{\partial }{\partial s}  \int_{-\infty}^a p_s^{+(\nu)} (t,r)dr = \,-\, \frac{\partial}{\partial s} \int_a^t p_s^{+(\nu)} (t,r)dr.  
    \end{equation}
    \end{proposition}
We shall also need the joint distribution of  $T_{t}^{a+*(\nu)}(s)$ and $\tau_a^{t,(\nu)}$  for any  $s\ge 0$. Notice that for any $ a \le r < t $, \begin{align*}
\mathbf{P}\left [ T_{t}^{a+*(\nu)}(s) > r,\,\tau_a^{t,(\nu)} > \xi \,\right ]  =  \mathbf{P} \left [ T_{t}^{a+*(\nu)}(s) > r, \, T_{t}^{a+*(\nu)}(\xi) > a\right ].  
\end{align*}
Moreover, $\xi \le s$ implies
\begin{equation} \nonumber %\label{firstJoint}
\mathbf{P}\left [T_{t}^{a+*(\nu)} (s)> r, \,  T_{t}^{a+*(\nu)}(\xi) > a\right ] = \mathbf{P}\left [T_{t}^{a+*(\nu)} (s)> r\right ],Ê
\end{equation}
whilst  for $s < \xi$,
\begin{align} \nonumber 
\mathbf{P}\left [T_{t}^{a+*(\nu)} (s)> r, \, T_{t}^{a+*(\nu)}(\xi) > a \right ]  &= \int_r^t p_s^{a+*(\nu)}(t,w) \left ( \int_a^w p_{\xi-s}^{a+*(\nu)} (w,y)dy\right ) dw \\
&= \int_r^t p_s^{+(\nu)}(t,w) \left ( 1 - \int_{-\infty}^a p_{\xi-s}^{+(\nu)} (w,y)dy \right ) dw.
\nonumber %\label{secondJoint}
\end{align}
Therefore, 
\begin{align} 
\varphi_{s,a}^{t,(\nu)} (r,\xi):= \frac{\partial^2}{\partial \xi \partial r} \mathbf{P}\left [ T_{t}^{a+*(\nu)}(s) \le  r,\, \tau_a^{t,(\nu)} \le \xi \right ] = \frac{\partial^2}{\partial \xi \partial r}  \mathbf{P} \left [T_{t}^{a+*(\nu)}(s) > r, \,  \tau_a^{t,(\nu)} > \xi \right ],  \nonumber 
\end{align}
yields the next result.
\begin{proposition}\label{jointP}
 Suppose the conditions (H0)-(H3) hold.  Then, for any $s \ge 0$ and $t \in (a,b]$,   the joint distribution  of  the pair $\left (T_{t}^{a+*(\nu)} (s), \tau_a^{t,(\nu)}\right )$, denoted by $\varphi_{s,a}^{t,(\nu)} (dr,d\xi)$, has the density given by
\begin{align} 
\varphi_{s,a}^{t,(\nu)}(r,\xi) :&= 
\, \, \mathbf{1}_{\{s< \xi\}} p_s^{+(\nu)}(t,r) \frac{\partial}{\partial \xi} \int_{-\infty}^a p_{\xi-s}^{+(\nu)}(r,y) dy, &  \,\, a \le r < t. \label{densityJoint}
\end{align}
\end{proposition}
\begin{remark}
Since the processes $T_t^{a+(\nu)}$,  $T_t^{a+*(\nu)}$ and $T_t^{+(\nu)}$ coincide before the first exit time $\tau_a^{t,(\nu)}$, the equation (\ref{densityJoint}) provides the joint density of $(T_{t}^{a+(\nu)} (s), \tau_a^{t,(\nu)})$  and   $(T_{t}^{a+(\nu)} (s), \tau_a^{t,(\nu)})$, for any $s \ge 0$ and for  $s < \xi$, respectively.
\end{remark}
By definition of the generator of a Feller process (see, e.g. \cite{EK}, \cite{KV0}),   if $S_s^{a+*(\nu)}$  is the semigroup of the stopped process $T_t^{a+*(\nu)}$, then $u \in \mathfrak{D}_{a+*}^{(\nu)}$ if,  and only if,
\begin{equation}\nonumber
-D_{a+*}^{(\nu)} u = \lim_{s\downarrow 0}\frac{ S_s^{a+*(\nu)} u - u}{s},
\end{equation}
where the limit is in the sense of the norm in $C[a,b]$. Analogously, if $S_s^{a+(\nu)}$ is the semigroup of  the killed process $T_t^{a+(\nu)}$, then $u \in \mathfrak{D}_{a+}^{(\nu)}$ if,  and only if, 
\begin{equation}\nonumber
-D_{a+}^{(\nu)} u = \lim_{s\downarrow 0}\frac{ S_s^{a+(\nu)} u - u}{s},
\end{equation}
where  the limit is in the sense of the norm in  $C_a[a,b]$.

Let us now introduce an operator which will play and important role to characterize the  domain of the generators $-D_{a+*}^{(\nu)}$ and $-D_{a+}^{(\nu)}$. 

For $\lambda \ge 0$,  define   
\begin{equation}\label{defM}
 M_{a,\lambda}^{+(\nu)} g(t)  := \mathbf{E} \left [  \int_0^{\tau_a^{t, (\nu)}} e^{-\lambda s} g \left (T_t^{+(\nu)} (s)\right) ds\right ],\quad t \in (a,b],  
 \end{equation}
  for  any (non constant) function $g \in B[a,b]$,   and  
\begin{equation}\label{defM2}
 M_{a,\lambda}^{+(\nu)} 1(t)  := \mathbf{E} \left [  \int_0^{\tau_a^{t, (\nu)}} e^{-\lambda s}  ds\right ],\quad t \in [a,b],  
 \end{equation}
 when  $g(t)  \equiv 1(t)$ (the constant function 1). 
 Then
\begin{equation}\label{M1}
  M_{a,\lambda}^{+(\nu)} \cdot 1 (t)=  \frac{1}{\lambda} \left ( 1 - \mathbf{E} \left[ e^{-\lambda \tau_a^{t, (\nu)}}\right ]\right ), 
\end{equation}
implying
\begin{equation}\nonumber %\label{laplaceM}
  \mathbf{E} \left[ e^{-\lambda \tau_a^{t, (\nu)}}\right ] =  1 - \lambda M_{a,\lambda}^{+(\nu)} \cdot 1 (t). 
\end{equation}
Further,  
 \begin{equation}
   \nonumber
 M_{a,\lambda}^{+(\nu)} c  = c  M_{a,\lambda}^{+(\nu)} \cdot 1(t),   \quad t \in [a,b],
 \end{equation}
for any constant function equals to $c$ (we shall use it mainly for the constant $g(a)$).
Note that the   stochastic continuity of the process $T_t^{a+*(\nu)}$ implies that  $M_{a, \lambda}^{(\nu)}g(\cdot)$ is  continuous on $[a,b]$. Moreover,
\[ | M_{a,\lambda}^{+(\nu)} g(t)| \le ||g|| \sup_{t \in [a,b]} \mathbf{E} \left[ \tau_a^{t,(\nu)}\right ].\]
\begin{lemma}\label{MLaplace-MJoint}
Suppose that $\nu$ satisfies  the conditions (H0)-(H3). Then
\begin{equation}\label{MLaplace}
\mathbf{E} \left[ e^{-\lambda \tau_a^{t, (\nu)}}\right ]  = \int_0^{\infty} e^{-\lambda s} \left(   \frac{\partial }{\partial s}  \int_{-\infty}^a p_s^{+(\nu)} (t,r)dr \right)ds, \quad t \in (a,b];
\end{equation}
and for any $g \in B[a,b]$
\begin{equation}\label{MJoint}
M_{a,\lambda}^{+(\nu)} g(t)  =   \int_0^{t-a} g(t-r) \int_0^{\infty} e^{-\lambda s}    p_s^{+(\nu)}(t,t-r)  \, ds\, dr, \quad t \in (a,b].  \end{equation}
\end{lemma}
\begin{proof}
Equality (\ref{MLaplace}) follows directly by using the density function $\mu_a^{t,(\nu)}$ of the r.v. $\tau_a^{t, (\nu)}$ as given in (\ref{lawofTau}).   To prove  (\ref{MJoint}), observe that Fubini's theorem allows one to rewrite $M_{a,\lambda}^{+(\nu)} g(t)$ as
\[ M_{a,\lambda}^{+(\nu)} g(t)  = \int_0^{\infty} e^{-\lambda s} \mathbf{E} \left[  \mathbf{1}_{\{\tau_a^{t, (\nu)} > s\}}  g\left ( T_t^{+(\nu)}(s)\right)\right]ds. \]
Using (\ref{densityJoint}), i.e.,  the joint density $\varphi_{s,a}^{t,(\nu)}(r,\xi)$   of the process $\left (T_t^{+(\nu)}(s), \tau_a^{t,(\nu)} \right )$ for $s < \xi$, yields
\begin{align*}
M_{a,\lambda}^{+(\nu)} g(t) &= \int_0^{\infty} e^{-\lambda s}   \left [ \int_a^t \int_0^{\infty} \mathbf{1}_{\{\xi > s\}}  g\left (r\right) \varphi_{s,a}^{t,(\nu)}(r,\xi) \,d\xi \,dr  \right]ds \\
&=  \int_0^{\infty} e^{-\lambda s}   \int_a^t  g(r) p_s^{+(\nu)}(t,r) \int_s^{\infty}   \left ( \frac{\partial}{\partial \xi} \int_{-\infty}^a p_{\xi-s}^{+(\nu)}(r,y) dy \right ) \, d\xi \, dr\, ds \\
&=  \int_0^{\infty} e^{-\lambda s}   \int_a^t  g(r) p_s^{+(\nu)}(t,r) \int_0^{\infty}   \left ( \frac{\partial}{\partial \gamma} \int_{-\infty}^a p_{\gamma}^{+(\nu)}(r,y) dy \right ) \, d\gamma \, dr\, ds \\
&=  \int_0^{\infty} e^{-\lambda s}   \int_a^t  g(r) p_s^{+(\nu)}(t,r) \int_0^{\infty}  \mu_a^{r,(\nu)} (\gamma) d\gamma  \, dr\, ds   \\
&=  \int_0^{\infty} e^{-\lambda s}   \int_a^t  g(r) p_s^{+(\nu)}(t,r)  \, dr\, ds,  
\end{align*}
where the last equality holds as $\mu_a^{r, (\nu)}$ is the density function of the r.v. $\tau_a^{r, (\nu)}$.   This  implies the result by another interchange in the order of integration and by a change of variable.
\end{proof}
\begin{remark}
Equality (\ref{MLaplace}) can be written as
\begin{equation}\label{secondLaplace}
\mathbf{E} \left[ e^{-\lambda \tau_a^{t, (\nu)}}\right ]  = \lambda \int_0^{\infty} e^{-\lambda s} \left( \int_{-\infty}^a p_s^{+(\nu)}(t,r)dr \right)ds,\quad  t \in (a,b], 
\end{equation}
which follows by integration by parts. 
\end{remark}
Let us now define the following space of functions:  
\begin{equation}\label{Ma}
\mathfrak{M_{a, \lambda}^{+(\nu)}} := \left \{ \,\,u\,\, \Big | u (t) = c M_{a,\lambda}^{+(\nu)} \cdot 1 (t) + d; \,\,\, t \in [a,b],\,\, c,d \in \mathbb{R}  \right\}.
\end{equation}
\begin{lemma}\label{domains}
Let $\nu$ be a function satisfying conditions  (H0)-(H1). If  $\lambda >0$,  then 
\begin{equation} \nonumber
\mathfrak{D}_{a+*}^{(\nu)} = \left\{ u_g \,\Big |\, u_g(t) = g(a)\frac{1}{\lambda } \left( 1 - \lambda M_{a,\lambda}^{+(\nu)} \cdot 1 (t) \right ) + M_{a,\lambda}^{+(\nu)} g(t),  \quad g \in C[a,b]    \right \},
\end{equation} 
and
\begin{equation}\nonumber
\mathfrak{D}_{a+}^{(\nu)}= \left\{ w_g \,\Big |\, w_g(t) =  M_{a,\lambda}^{+(\nu)} g(t),  \quad g \in C_a[a,b] \right \}.
\end{equation}
Further, if $\nu$ also satisfies (H2)-(H3), then equalities  (\ref{MLaplace}) and (\ref{MJoint}) give  explicit expressions for $\left( 1 - \lambda M_{a,\lambda}^{+(\nu)} \cdot 1 (t) \right )$ and  $M_{a,\lambda}^{+(\nu)} g(t)$, respectively. 
\end{lemma}
\begin{proof}
Let us take any $u \in \mathfrak{D}_{a+*}^{(\nu)}$. Since $-D_{a+*}^{(\nu)}$ is the generator of a Feller process on $C[a,b]$,  Theorem 1.1 in  \cite{dynkin1965} implies the existence of  a function $g \in C[a,b]$ such that $u = R_{\lambda}^{a+*(\nu)} g$. By definition of the resolvent and by Fubini's theorem 
\begin{align*}
u(t)&= \mathbf{E} \left[   \int_0^{\infty} e^{-\lambda s}  g(T_t^{a+*(\nu)} (s))ds  \right] = \mathbf{E} \left [  \left(  \int_0^{\tau_a^{t,(\nu)}} + \int_{\tau_a^{t,(\nu)}}^{\infty} \right) e^{-\lambda s}  g(T_t^{a+*(\nu)} (s))ds,  \right], \end{align*}
where $\mathbf{E}\left [\tau_a^{t, (\nu)} \right ] < +\infty$  by Lemma \ref{propH0}. 

Since the process is stopped at time $\tau_a^{t, (\nu)}$, $g\left (T_t^{a+*(\nu)} (s)\right ) = g(a)$ for all $s \ge \tau_a^{t, (\nu)}$. Moreover, before time $\tau_a^{t, (\nu)}$ the processes $T_t^{a+*(\nu)}$ and $T_t^{+(\nu)}$ coincide. Therefore,
\begin{align}\nonumber
u(t) &= g(a) \mathbf{E} \left [  \int_{\tau_a^{t,(\nu)}}^{\infty} e^{-\lambda s}  ds  \right] + \mathbf{E} \left [   \int_0^{\tau_a^{t,(\nu)}} e^{-\lambda s}  g(T_t^{+(\nu)} (s))ds  \right] \\ \nonumber
&=g(a) \left \{\frac{1}{\lambda} -  \mathbf{E} \left [ \int_0^{\tau_a^{t,(\nu)}} e^{-\lambda s}  ds  \right] \right\} + \mathbf{E} \left [   \int_0^{\tau_a^{t,(\nu)}} e^{-\lambda s}  g(T_t^{+(\nu)} (s))ds  \right] \\ \label{desarrollo}
&=g(a)\frac{1}{\lambda } \left( 1 - \lambda M_{a,\lambda}^{+(\nu)} \cdot 1 (t) \right ) + M_{a,\lambda}^{+(\nu)} g(t),
\end{align}
as required. 
The characterization of the domain  $\mathfrak{D}_{a+}^{(\nu)}$ is similar to the previous case. Take any $w \in \mathfrak{D}_{a+}^{(\nu)}$, then there exists a function $g \in C_a[a, b]$ such that $w = R_{\lambda}^{a+(\nu)} g$. Hence, a similar procedure  yields (\ref{desarrollo}) and $g(a) = 0$ implies
\[  R_{\lambda}^{a+(\nu)} g(t) = M_{a,\lambda}^{+(\nu)} g(t).\]
Finally, observe that under assumptions (H2)-(H3),  Lemma \ref{MLaplace-MJoint} holds.  
\end{proof}
Let us now see how the resolvents (and hence the domains) of the stopped and killed processes are related to.   
\begin{lemma}\label{relation}
Let $\nu$ be a function satisfying condition (H0). Suppose  $\lambda > 0$ and $g\in C[a, b]$. Define $\tilde{g} (t) = g(t)- g(a)$, then
 \begin{align*}
R_{\lambda}^{a+(\nu)} \tilde{g}(t) = R_{\lambda}^{a+*(\nu)} \tilde{g}(t)  = R_{\lambda}^{a+*(\nu)} g(t) -g(a) R_{\lambda}^{a+*(\nu)} \cdot 1(t), \end{align*}
and
 \begin{equation}\label{repM}
R_{\lambda}^{a+(\nu)} \tilde{g}(t) = M_{a,\lambda}^{+(\nu)} g(t) - g(a)M_{a,\lambda}^{+(\nu)}\cdot 1(t). 
\end{equation}
In particular,  $R_{\lambda}^{a+(\nu)} \tilde{g}(t)$ belongs to both domains $\mathfrak{D}_{a+*}^{(\nu)}$ and $\mathfrak{D}_{a+}^{(\nu)} $.
\end{lemma}
\begin{proof}
Follows directly from the linearity of the operators $R_{\lambda}^{a+*(\nu)}$ and $M_{a,\lambda}^{(\nu)}$, and by using  that $\tilde{g}(a) = 0$. 
\end{proof}
\begin{remark}
Let us stress that  Lemma \ref{relation} implies  that  $M_{a,\lambda}^{+(\nu)}g$ coincides with the resolvent $R_{\lambda}^{a+(\nu)}g$ only when the function $g(a) = 0$. Hence, only in this case $M_{a,\lambda}^{+(\nu)}g$ belongs to the domain of both generators $-D_{a+*}^{(\nu)}$ and $-D_{a+}^{(\nu)}$.   
 \end{remark}
It is worth noting that Theorem 1.1 in \cite{dynkin1965} also guarantees that for  $ g \in C[a, b]$ and $\lambda > 0$, the function  $u_g:=R_{\lambda}^{a+*(\nu)} g$ is the unique  \textit{solution in the domain} of $-D_{a+*}^{(\nu)}$ to 
\begin{equation}\label{udomain}
-D_{a+*}^{(\nu)} u(t) = \lambda u(t) - g(t), \quad t \in [a,b].
\end{equation}
 Similarly, if $g \in C_a[a,b]$, then  $w = R_{\lambda}^{a+(\nu)}g$ is the unique \textit{solution in the domain} of $-D_{a+}^{(\nu)}$ to
\begin{equation}\nonumber
-D_{a+}^{(\nu)} w(t) = \lambda w(t) - g(t), \quad t \in [a,b].
\end{equation}
 \section{Equations involving $D-$ and $D_{*}$-  operators}
The probabilistic representation of solutions (in the generalized sense) to linear equations involving $D$- and $D_{*}$-operators will be studied in this section. 
\subsection{Linear equations involving $D$-operators}
Consider the problem of finding a continuous function $u$ on  $[a,b]$ satisfying
 \begin{equation}\label{eqlinear10}
-D_{a+}^{(\nu)} w(t)  =  \lambda  w(t) - g(t), \quad t \in [a,b], \quad \quad w(a) = w_a,
\end{equation}
for $\lambda \ge 0$, $g \in B[a,b]$ and $w_a = 0$. Hereafter, we shall refer to (\ref{eqlinear10}) as the RL type  problem $(-D_{a+}^{(\nu)}, \lambda, g, w_a)$, wherein we shall  always take  $w_a=0$.  Similar notation will be used for the linear equation with the corresponding Caputo type operator: $(-D_{a+*}^{(\nu)}, \lambda, g, w_a)$ for any $w_a \in \mathbb{R}$.
\begin{definition}\label{D:GRL}
Let $g \in B[a,b]$ and $\lambda \ge 0$. A function  $w \in C_a[a,b]$ is said to solve the RL type problem  $(-D_{a+}^{(\nu)}, \lambda,g, 0)$ as
 \begin{itemize} \item [(i)]   a \emph{solution in the domain of the generator}  if $w$ satisfies (\ref{eqlinear10}) and belongs to the domain of the generator $-D_{a+}^{(\nu)}$; \item  [(ii)]  a \emph{generalized solution} if for all sequence of functions $g_n \in C_a[a,b]$ such that $\sup_n ||g_n|| < \infty$ uniformly on $n$, and $\lim_{n\to \infty} g_n \to g$ a.e.,  it holds that  $w(t) = \lim_{n\to \infty} w_n(t)$ for all $t \in [a,b]$, where $w_n$ is the solution (in the domain of the generator)  to the RL problem $(-D_{a+}^{(\nu)}, \lambda, g_n, 0)$.\end{itemize}
\end{definition}

\begin{remark} From this definition it follows that if there exists a generalized solution, then this is unique.
\end{remark}
\begin{definition} The RL type equation (\ref{eqlinear10})   is  \textit{well-posed in the generalized sense} if  it has a unique generalized  solution.
\end{definition}

\noindent \textbf{Well-posedness result for the RL type linear equation.}
\begin{theorem} (\textbf{Case $\lambda > 0$}) \label{wellposedness1}
Let $\nu$ be a function satisfying conditions (H0)-(H1) and assume  $\lambda > 0$.    
\begin{itemize}
\item [(i)] If $g\in C_a[a,b]$, then the  linear problem   $(-D_{a+}^{(\nu)}, \lambda, g,0)$  has a unique solution in the domain of the generator   given by  $w = R_{\lambda}^{a+(\nu)} g$ (the resolvent operator at $\lambda$).

\item [(ii)] For any $g \in B[a,b]$,   the linear equation   $(-D_{a+}^{(\nu)}, \lambda, g,0)$ is well-posed in the generalized sense and   the solution admits the stochastic representation
\begin{equation}\label{MsolRL}
w(t) =  \mathbf{E} \left[Ê \int_0^{\tau_a^{t,(\nu)}}  e^{-\lambda s} g \left( T_t^{+(\nu)} (s) ds \right )\right ].
\end{equation}
Moreover, if additionally $\nu$ satisfies conditions (H2)-(H3), then
\begin{equation}\label{solTheoremRL}
w(t) =  \int_0^{t-a} g(t-r)  \int_0^{\infty} e^{-\lambda s} p_s^{+(\nu)}(t,t-r)  \,ds\,\, dr.
\end{equation}
\item [(iii)] If $g\in C[a,b]$, then the solution  to (\ref{eqlinear10}) belongs to $\mathfrak{D}_{a+}^{(\nu)} \oplus \mathfrak{M}_{a,\lambda}^{+(\nu)}$ (the direct sum of the domain of the generator $-D_{a+}^{(\nu)}$ and the space defined in (\ref{Ma})). 
\end{itemize}
 \end{theorem}
\begin{proof}
 (i)   Take $g \in C_a [a,b]$. Since $g(a)=0$, $w(a) = 0$ and $\lambda > 0$,  Theorem 1.1 in \cite{dynkin1965} implies directly that $w(t) =R_{\lambda}^{a+(\nu)}g(t)$ is the unique solution to (\ref{eqlinear10})  belonging to the domain of the generator.  Moreover,   Lemma \ref{domains} implies 
\begin{equation}\nonumber 
w(t) = M_{a, \lambda}^{+(\nu)} g(t) = \mathbf{E} \left[  \int_0^{\tau_a^{t,(\nu)}} e^{-\lambda s}  g \left(T_t^{+(\nu)} (s)  \right ) ds \right  ].
\end{equation}
\noindent (ii) Let  us now take  any function  $g \in B[a, b]$. As $g$ does not necessarily belong to  $C_a[a,b]$, the resolvent operator no longer provides a solution to (\ref{eqlinear10}). However, using  Definition \ref{D:GRL} we will see that there exists a unique \textit{generalized solution}. To do this, take any sequence  $g_n \in C_a[a,b]$ such that $\lim_{n\to\infty} g_n = g$ a.e. and $\sup_n ||g_n|| < \infty$ uniformly on $n$.   The  procedure consists in finding the generalized  solution as a limit of solutions to the equations
\begin{equation}\nonumber 
-D_{a+}^{(\nu)} w_n(t)  =  \lambda  w_n(t) - g_n(t), \quad t \in (a,b], \quad \quad w_n(a) = 0.
\end{equation}
Since each $g_n\in C_a[a,b]$, the previous case guarantees the existence of a unique solution $w_n  \in \mathfrak{D}_{a+}^{(\nu)}$ given by
\begin{equation}\nonumber
w_n(t) = \mathbf{E} \left[  \int_0^{\tau_a^{t,(\nu)}} e^{-\lambda s}  g_n \left(T_t^{+ (\nu)} (s)  \right ) ds \right  ].  
\end{equation}
Using that   $||g_n||$ is uniformly  bounded, the dominated convergence theorem (DCT) implies 
% THESIS
%\[ \lim_{n \to \infty } \textbf{E} \left [ \textbf{1}_{\{ \tau_a^{t,(\nu)} > s\} }g_n \left(T_t^{+ (\nu)} (s)  \right ) \right ]  = \textbf{E} \left [ \textbf{1}_{\{ \tau_a^{t,(\nu)} > s\} } g \left(T_t^{+ (\nu)} (s)  \right ) \right ].
%\]
%Further, as $\left | e^{-\lambda s}  \textbf{E} \left [ g_n \left(T_t^{+ (\nu)} (s)  \right ) \right ]   \right | \le  \,C \,e^{-\lambda s} $, then DCT  also yields
\begin{equation}\nonumber
\lim_{n\to \infty} w_n(t) = \mathbb{E} \left[   \int_0^{\tau_a^{t,(\nu)}} e^{-\lambda s}  g\left(T_t^{+ (\nu)} (s)  \right ) ds \right  ]  =: w(t). 
\end{equation}
Observe that the continuity of $w (t)$ on $(a, b]$ is a consequence of the continuity of the mapping $t \to \mathbf{E} \left [g\left(  T_t^{+(\nu)} (s)\right)\right ]$. The continuity at   $t=a$ holds by  the regularity in expectation of $\tau_a^{t,(\nu)}$ (a consequence of assumption (H1)).  Therefore,  $w \in C_a[a, b]$ is a generalized solution  to the linear equation (\ref{eqlinear10}). Finally, the representation in  (\ref{solTheoremRL}) follows directly from  Lemma  \ref{MLaplace-MJoint}.

\vspace{0.2cm}
\noindent (iii)   To prove that $ w \in \mathfrak{D}_{a+}^{(\nu)} \oplus \mathfrak{M}_{a,\lambda}^{(\nu)}$ whenever $g\in C[a,b]$, we use the equality (\ref{repM}) in Lemma \ref{relation} to obtain
\[ M_{a,\lambda}^{+(\nu)} g(t) = R_{\lambda}^{a+(\nu)} \hat{g}(t) + g(a)M_{a,\lambda}^{+(\nu)}\cdot 1(t), \]
where $\hat{g}(t) = g(t)- g(a)$, implying  the result.
\end{proof}

\begin{theorem} \textbf{(Case $\lambda = 0$)}\label{wellposedness10}
Theorem \ref{wellposedness1} with $\lambda=0$ is valid for the equation  
\begin{equation}\label{zeroL}
-D_{a+}^{(\nu)} w(t) = -g(t),\quad t \in(a,b]; \quad \quad w(a) = 0.
\end{equation}
\end{theorem}
\begin{proof}
It follows from the same arguments used for  $\lambda >0$,  so that we skip the details. Let us just notice that, since 
\begin{align*} 
|R_{0}^{a+(\nu)} g(t) |& \le \mathbf{E} \left[  \int_0^{\tau_a^{t,(\nu)}} |g(T_t^{a+(\nu)} (s))|ds \right ] \le ||g|| \sup_{t \in(a,b]} \mathbf{E}\left [\tau_a^{t,(\nu)}\right], 
\end{align*}
  Lemma \ref{propH0} implies that the  \textit{potential operator}  corresponding to $T_t^{a+(\nu)}$ is bounded. Thus, 
Theorem 1.1' in \cite{dynkin1965}  ensures that the proof in \ref{wellposedness1} remains true for $\lambda >0$ if one replaces the resolvent operator $R_{\lambda}^{a+(\nu)}$ by the potential operator $R_0^{a+(\nu)}$.   
Also observe that   $w \in \mathfrak{D}_{a+}^{(\nu)} \oplus \mathfrak{M}_{a,0}^{+(\nu)}$ whenever  $g \in C[a,b]$ as $w$ can be written as
\[w (t) = R_0^{a+(\nu)}\tilde{g}(t) + g(a) \mathbf{E}\left[\int_0^{\tau_a^{t,(\nu)}}ds \right], \]
where $\tilde{g}(t) := g(t) -g(a)$, for all $t \in [a,b]$. 
\end{proof}

\subsection{Linear equations involving $D_{*}$-operators}
Let $a \in \mathbb{R}$ and  $\lambda \ge 0$. Consider the problem of finding a function  $u\in C[a,b]$ solving
\begin{equation}\label{eqlinear1C}
-D_{a+*}^{(\nu)} u(t)  =  \lambda  u(t) - g (t), \quad t \in (a,b],\quad u(a)=u_a,
\end{equation}
for a given function $g$ on $[a, b]$.

If $u$ belongs to the domain of the generator $-D_{a+*}^{(\nu)}$, then the linear equation (\ref{eqlinear1C}) can be written  in terms of the  RL type operator $D_{a+}^{(\nu)}$ as follows.  Define $w(t) := u(t) - u_a$ for all $t \in[a,b]$, then   $-D_{a+*}^{(\nu)} w(t) = -D_{a+*}^{(\nu)} u(t)$ as $-D_{a+*}^{(\nu)} u_a = 0$.  Setting $\tilde{g}(t) := g(t) - \lambda u_a$,  it follows that 
\begin{equation}\label{2zero2v2}
-D_{a+}^{(\nu)} w(t)  =  \lambda w(t) - \tilde{g}(t),   \quad t\in (a,b] \quad w(a) = 0.
\end{equation}
Hence, $u(t) = w(t) +u_a$ is a solution to the original problem if, and only if, $w$ solves (\ref{2zero2v2}). The previous leads to the following definition.
\begin{definition}\label{Def-GenC}
Let $g \in B[a,b]$ and $\lambda \ge 0$.  A function $u$ in $C[a,b]$   is said to solve the  Caputo type  problem $(-D_{a+*}^{(\nu)}, \lambda,g, u_a)$ as 
\begin{itemize}
\item [(i)] a \emph{solution in the domain of the generator}  if $w$ satisfies (\ref{eqlinear1C}) and belongs to the domain of the generator $-D_{a+*}^{(\nu)}$; 
\item [(ii)] a \textit{generalized solution}   if $u(t) = u_a + w(t)$ for all $t \in [a,b]$, where $w$ is the (possibly generalized)  solution  to the RL type problem $(-D_{a+}^{(\nu)},\lambda, g-\lambda u_a, 0)$.
\end{itemize}
\end{definition}
\begin{remark}
Definition \ref{Def-GenC}, $(ii)$, is given in terms of the RL type solution, but it can also be written in terms of the approximation of solutions belonging to the domain of the corresponding generator. 
\end{remark}
\begin{definition}
The  Caputo type equation (\ref{eqlinear1C}) is \textit{well-posed in the generalized sense}  if it has a unique generalized solution and it depends continuously on the initial condition.
\end{definition}

\noindent \textbf{Well-posedness result for the Caputo type linear equation.}

\begin{theorem} \textbf{(Case $\lambda > 0$)}\label{wellposedness2}
Let $\nu$ be a function satisfying conditions (H0)-(H1) and suppose $\lambda > 0$.  
\begin{itemize}
\item [(i)] If  $g \in C[a,b]$ and $g(a)=\lambda u_a$, then the linear equation $(-D_{a+*}^{(\nu)},\lambda, g, u_a)$ has a unique solution in the domain of the generator    given by $u = R_{\lambda}^{a+*(\nu)} g$ (the resolvent operator at $\lambda$).
\item[(ii)] For any $g \in B[a,b]$ and  $u_a\in \mathbb{R}$, the   linear equation  $(-D_{a+*}^{(\nu)},\lambda, g, u_a)$  is well-posed in the generalized sense and   the solution  admits the stochastic representation
\begin{equation}\label{SRCaputo}
u(t) = u_a  \mathbf{E} \left[ e^{-\lambda \tau_a^{t,(\nu)}}\right]  + \mathbf{E} \left[Ê \int_0^{\tau_a^{t,(\nu)}}  e^{-\lambda s} g \left( T_t^{+(\nu)} (s) ds \right )\right ].
\end{equation}

\noindent Moreover, if additionally $\nu$ satisfies conditions (H2)-(H3),  then 
\begin{equation}\label{solTheoremC}
u(t) =  u_a \int_0^{\infty} e^{-\lambda s} \mu_a^{t,(\nu)} (s)ds + \int_0^{t-a} g(t-r)  \int_0^{\infty} e^{-\lambda s} p_s^{+(\nu)}(t,t-r)  \,ds\,  \, dr,
\end{equation}
where $\mu_a^{t,(\nu)}(s)$ denotes the density  function of the r.v. $\tau_a^{t,(\nu)}$.
\item[(iii)] If $g\in C[a,b]$, then the solution to (\ref{eqlinear1C}) belongs to $ \mathfrak{D}_{a+*}^{(\nu)} \oplus \mathfrak{M}_{a,\lambda}^{(\nu)}$ (the direct sum of the domain of the generator $-D_{a+*}^{(\nu)}$ and the space defined in (\ref{Ma})).  
 \end{itemize}
 \end{theorem}
\begin{proof}
(i)  Since $-D_{a+*}^{(\nu)}$ is the generator of a Feller process on $C[a,b]$, for any $g \in C[a,b]$ the function $u(t) := R_{\lambda}^{a+*(\nu)}g(t)$ is the unique solution to (\ref{udomain}) which belongs to the domain of the generator. This implies that  $u(a) = R_{\lambda}^{a+*(\nu)}g(a) = g(a)/\lambda$.  Hence,  condition $g(a) = \lambda u_a$  ensures that $u$ satisfies the boundary condition in (\ref{eqlinear1C}), as required.

\vspace{0.2cm}
\noindent (ii) By Definition \ref{Def-GenC}, $u$ is the generalized solution if $u(t) =  w(t) + u_a$, where $w$ is the solution to the RL type problem  $(-D_{a+}^{(\nu)},\lambda, g(t) - \lambda u_a, 0)$ whose well-posedness is guaranteed  by  Theorem \ref{wellposedness1}.

Moreover, the  equality  (\ref{MsolRL}) and the linearity of $M_{a, \lambda}^{+(\nu)}$  yield
\begin{align*}
w(t) &=   M_{a, \lambda}^{+(\nu)} g(t) -\lambda u_a M_{a, \lambda}^{+(\nu)} \cdot 1(t) = M_{a, \lambda}^{+(\nu)} g(t) - u_a  \left ( 1 - \mathbf{E} \left[ e^{-\lambda \tau_a^{t, (\nu)}}\right ]\right ), 
\end{align*}
where the last equality holds due to equation  (\ref{M1}).  Thus, (\ref{SRCaputo}) is obtained by plugging the previous expression for $w(t)$ into $u(t) = w(t) + u_a$.  This  representation  implies directly  the continuity on the initial condition $u_a$, as required for the well-posedness. Finally, the explicit solution  (\ref{solTheoremC}) follows from  Lemma \ref{MLaplace-MJoint}.

\vspace{0.2cm}
\noindent (iii) Assume now that $g\in C[a,b]$.   Since $u(t) =  M_{a, \lambda}^{+(\nu)} g(t) -\lambda u_a M_{a, \lambda}^{(\nu)} \cdot 1(t) + u_a$, by linearity one can rewrite it as
\begin{align*}
u(t) &= M_{a, \lambda}^{+(\nu)} [g -g(a) + g(a)](t)   - \lambda u_a M_{a, \lambda}^{+(\nu)} \cdot 1(t) + u_a \\
&=  R_{\lambda}^{a+*(\nu)} [g -g(a)](t)  + [g(a) - \lambda u_a] M_{a, \lambda}^{+(\nu)}\cdot 1(t) + u_a.
\end{align*}
We then conclude that $u \in \mathfrak{D}_{a+*}^{(\nu)} \oplus \mathfrak{M}_{a,\lambda}^{+(\nu)}$ as  \[R_{\lambda}^{a+*(\nu)} [g -g(a)] \,\,\in\,\,  \mathfrak{D}_{a+*}^{(\nu)},\] and \[ [g(a) + \lambda u_a] M_{a, \lambda}^{+(\nu)}\cdot 1(t) + u_a\,\,\in\,\, \mathfrak{M}_{a,\lambda}^{+(\nu)}. \]
\end{proof}

\begin{theorem}(\textbf{Case $\lambda = 0$})\label{wellposedness10C}
All assertions in Theorem \ref{wellposedness2}  with $\lambda=0$ hold for the equation
\begin{equation}\label{zeroC}
-D_{a+*}^{(\nu)} u(t) = - g(t),\quad t \in (a,b]; \quad \quad u(a) = u_a.
\end{equation}
\end{theorem}
\begin{proof}
Since the problem (\ref{zeroC}) 
 rewrites  as
\begin{equation}\label{zeroRLC}
-D_{a+}^{(\nu)} w(t) = \tilde{g}(t),\quad t \in (a,b]; \quad \quad w(a) = 0,
\end{equation}
with $w(t) = u(t) - u_a$ and $\tilde{g}(t) = g(t)$,  Theorem \ref{wellposedness10}   gives the potential operator $R_0^{a+(\nu)} \tilde{g}(t)$ as the solution to (\ref{zeroRLC}) for any $\tilde{g} \in C_a[a,b]$. Hence,  the unique generalized solution to (\ref{zeroC}) is given by $u(t) = u_a + \lim_{n\to \infty} R_0^{a+(\nu)}\tilde{g}_n(t)$ for any sequence  $\tilde{g}_n$  as in Definition \ref{D:GRL}.  Consequently, the same arguments used for $\lambda >0$ remain valid.
\end{proof}

\subsection{Fractional linear equations  involving  Caputo derivatives}

Since  Caputo derivatives are particular cases of the $D_{*}-$operators, the solution to fractional linear equations with the Caputo derivative  $D_{a+*}^{\beta}$, for $\beta \in (0,1)$, is obtained by a direct application of the previous results.  Namely,   Theorem \ref{wellposedness2} implies that the problem \begin{equation}
\nonumber 
  D_{a+*}^{\beta} u(t) = - \lambda u(t)+ g(t), \quad t\in(a,b],\quad u (a) = u_a \in \mathbb{R},
\end{equation}
 for a given $g \in B[a,b]$ and  $\lambda > 0$, has a unique generalized solution given by  
\begin{equation}\label{solutionC1B}
u (t)  = u_a \mathbf{E}\left  [ e^{-\lambda \tau_a^{t,\beta}} \right ] + \mathbf{E} \left [  \int_0^{\tau_a^{t,\beta}} e^{-\lambda s} g(T_{t}^{+\beta}(s))ds \right ], 
\end{equation}
where  $T_{t}^{+\beta}$ is  the inverted $\beta-$stable subordinator started at $t$.   
Moreover,  substituting (\ref{transitionpXInv}) into formula (\ref{solTheoremC}) yields 
\begin{align}
u(t)  =\, u_a \, & \frac{1}{\beta} (t-a) \int_0^{\infty} e^{\,-\,\lambda \, s}  \left (\,\,s^{-\frac{1}{\beta} - 1} w_{\beta} \left( (t-a) s^{-1/\beta};1,1 \right ) \right ) ds \,+  \nonumber \\  &+ \int_0^{t-a} g(t-r)   \left ( r^{\beta-1} \int_0^{\infty}  \exp\left \{-\lambda s r^{\beta}  \right \} s^{-1/\beta}   w_{\beta} (s^{-1/\beta} ;1,1)  ds\right )dr. \label{linearsolCB}
\end{align}
Further, if $g(a) = \lambda u_a$ and $g \in C[a,b]$, then $u$ belongs to the domain of the generator $-D_{a+*}^{(\nu)}$. %On the other hand, if $g\in C^1[a,b]$ (without any restriction on the boundary condition $u_a$), then $u$ is a smooth classical  solution, i.e.  $u \in C[a,b] \cap C^1 (a,b]$.  Finally, the conditions $g(a) = \lambda u_a$ and $g \in C^1[a,b]$  guarantee that the solution is smooth in the whole closed interval, i.e.  $u \in C^1[a,b]$.   The case $\lambda=0$ follows similarly from Theorem \ref{wellposedness10C}.  

\begin{corollary}\label{LaplaceTau}
Let $t \in (a,b]$ and  $\lambda > 0$. Then the Laplace transform of the first exit time from $(a, b]$ for the inverted $\beta-$stable subordinator started at $t$ is given by
\begin{align}\nonumber %\label{uno}
\textbf{E}[e^{-\lambda \tau_a^{t,\beta}}] = E_{\beta} (\,-\,\lambda (t-a)^{\beta}),
\end{align}
with $E_{\beta}$ denoting the Mittag-Leffler function  (see Preliminaries). Hence, 
\begin{equation}\nonumber %\label{mittag1}
E_{\beta} (\, - \,\lambda (t-a)^{\beta}) =  \frac{1}{\beta} (t-a) \int_0^{\infty} \exp  (\,-\,\lambda \, s) \,  s^{-\frac{1}{\beta} - 1} w_{\beta} \left( (t-a) s^{-1/\beta} ;1,1\right )   ds.
\end{equation}
\end{corollary}
\begin{proof}
Follows as  a consequence of formulas (\ref{General}) and  (\ref{linearsolCB}) by  taking   $g\equiv 0$ and $u_a=1$ and by uniqueness of solutions.
\end{proof}
 \begin{remark}
 Alternatively, Corollary \ref{LaplaceTau} can also be obtained by using the identity  (see \cite{zolotarev}, Theorem 2.10.2)
 \begin{equation}\label{identity}
\beta E_{\beta} (-u) = \int_0^{\infty} \exp (-uy) y^{-1-1/\beta} w_{\beta} (y^{-1/\beta};1,1)dy.
 \end{equation}
 Moreover, using this identity we can see that  (\ref{linearsolCB}) coincides with the well known solution given in (\ref{General}). 
 \end{remark}
%Therefore, the stochastic representation (\ref{solutionC1B}) and the uniqueness of solutions  provide a  pure probabilistic proof for  the known integral representation (\ref{identity}) for the Mittag-Leffler functions. Using the previous identity is also clear that  (\ref{linearsolCB}) coincides with the very-well known solution given in (\ref{General}) which is written  in  terms of the Mittag-Leffler functions  $E_{\beta}$ and $E_{\beta,\beta}$. %
 
 \section{Mixed linear equations}

In this section we study   linear  equations involving both the RL type and the Caputo type operators. The general setting will be  explained first in $\mathbb{R}^d$, and then we shall restrict ourselves   to the simplest 2-dimensional case. This is done to avoid cumbersome calculations which nevertheless     can be extended straightforward from the simple case analyzed here.

Let  $\textbf{a}=(a_1,\ldots, a_d), \textbf{b}=(b_1,\ldots, b_d)  \in \mathbb{R}^d$, such that $\textbf{a} < \textbf{b}$. The space $\mathbb{R}^d$ is assumed to be equipped with its natural partial order, the Pareto order, i.e. $\textbf{a} < \textbf{b}$ means $a_i < b_i$ for all $i =1, \ldots, d$. 
Notation $[\textbf{a},\textbf{b}]$ denotes the cartesian product $[a_1, b_1] \times \cdots \times [a_d,b_d]$ and $\textbf{t} \in [\textbf{a}, \textbf{b}]$ means $t_i \in [a_i,b_1]$ for all $i =1, \ldots, d$. Let us denote by $B[\textbf{a},\textbf{b}]$ and $C[\textbf{a},\textbf{b}]$ the space of bounded Borel measurable functions and continuous functions on  $[\textbf{a},\textbf{b}]$, respectively,  and by $C^1[\textbf{a},\textbf{b}]$ the space of continuous functions on $[\textbf{a},\textbf{b}]$ with continuous  first order partial derivatives on $[\textbf{a},\textbf{b}]$. Similar notation will be used  for  $(\textbf{a},\textbf{b}]$ and $(-\infty, \textbf{b}]$.

 Notation $\textbf{r}^{a_i}$  means a  vector $\textbf{r}=(r_1, \ldots, r_d) \in [\textbf{a}, \textbf{b}]$  having $r_i=a_i$ as its   $i$th-coordinate.  Since all the processes considered here have  decreasing sample paths, we will  be only interested in the boundary of $(\textbf{a}, \textbf{b}]$ given by $\textbf{r}^{a_i}$ for all $i=1,\ldots, d$. This subset  will be denoted by  
  \begin{equation}\nonumber %\label{boundaryI}
 \partial_{\textbf{a}}  (\textbf{a}, \textbf{b}]:=   \bigcup_{i=1}^d \big \{ \mathbf{r} \in [\textbf{a}, \textbf{b}]\,:\, \mathbf{r} = \mathbf{r}^{a_i}\,  \big \},    
 \end{equation}
 and  the space of continuous functions on $[\textbf{a},\textbf{b}]$ vanishing at the boundary $\partial_{\textbf{a}}(\textbf{a},\textbf{b}]$ will be denoted by  $C_{\textbf{a}}[\textbf{a},\textbf{b}]$.
 
  For  $i=1, \ldots, d$, let $\nu_i$ be a function satisfying conditions (H0)-(H1) and let $\textbf{t} \in (\textbf{a}, \textbf{b}]$. 
 The operator $- _{t_i}\!D_{a_i+}^{(\nu_i)}$ represents the $D$-operator  defined by $\nu_i$ acting  on the  variable $t_i$.  
For notational convenience set $\nu = (\nu_1, \ldots, \nu_d)$ and  define the \textit{mixed $D-$operator}  associated with the vector $\nu$ as
\begin{equation}\label{mixedL}
-\,\textbf{D}_{\textbf{a}+}^{(\nu)}:= -\, \sum_{i=1}^d  ~_{t_i}\!D_{a_i+}^{(\nu_i)},\quad \end{equation}
Hence, the operator $-\,\textbf{D}_{\textbf{a}+}^{(\nu)}$ is a sum of RL type operators each one acting  on a  different variable. Analogously, we define the  mixed operators $-\,\textbf{D}^{+(\nu)}$ and $-\,\textbf{D}_{\textbf{a}+*}^{(\nu)}$ by using  $D^{+(\nu_i)}$ and $~_{t_i}\!D_{a_i+*}^{(\nu_i)}$, respectively.

 Consider the RL type linear equation 
\begin{equation}\label{mixedP}
\begin{array}{rcll}
 -\,\textbf{D}_{\textbf{a}+}^{(\nu)}  w(\textbf{t}) &=& \lambda w(\textbf{t}) - g(\textbf{t}),& \textbf{t} \in (\textbf{a}, \textbf{b}],  \\
w (\textbf{t}) &=& 0, & \textbf{t}  \in \partial_\textbf{a} (\textbf{a}, \textbf{b}],  
\end{array}
\end{equation}
  for a given function  $g \in B[\textbf{a}, \textbf{b}]$ and $\lambda \ge 0$.
  
The operator $- \textbf{D}_{\textbf{a}+}^{(\nu)}$ can be thought of  as the generator of a Feller  process on $(\textbf{a}, \textbf{b}]$ obtained by killing the process generated by $-\,\textbf{D}^{+(\nu)}$ on an attempt to cross   the boundary  $\partial_{\textbf{a}} (\textbf{a}, \textbf{b}]$. The killed process (started at $\textbf{t}$)  will be denoted by   $ \textbf{T}^{\textbf{a}+(\nu)}_{\textbf{t}} = \left \{ \textbf{T}^{\textbf{a}+(\nu)}_{\textbf{t}} (s): s \ge 0 \right \}$.

Due to the independence,  the Lie-Trotter theorem \cite{EK, KV0} implies that  \[ \textbf{T}^{\textbf{a}+(\nu)}_{\textbf{t}} (s) = \left (T^{a_1+(\nu_1)}_{t_1} (s), \ldots, T^{a_d+(\nu_d)}_{t_d} (s)\right ),\] wherein each coordinate $T^{a_i+(\nu_i)}_{t_i}$ is an independent  process generated by $-_{t_i}\!D_{a_i+}^{(\nu_i)}$.  
 Let $\textbf{S}_s^{\textbf{a}+(\nu)} $  denote  the  semigroup of the process $\textbf{T}_{\textbf{t}}^{\textbf{a}+(\nu)}$ and let  $\hat{\mathfrak{D}}_{\textbf{a}+}^{(\nu)}$ be the domain of its  generator.   Then, $u \in \hat{\mathfrak{D}}_{\textbf{a}+}^{(\nu)}$ if, and only if, the limit 
\[ -\,\textbf{D}_{\textbf{a}+}^{(\nu)} u(\textbf{t}) = \lim_{s \to 0} \frac{\textbf{S}_s^{\textbf{a}+(\nu)} u(\textbf{t}) - u(\textbf{t})}{s}, \]
exists in the norm of $C_{\textbf{a}}[\textbf{a},\textbf{b}]$.

To solve (\ref{mixedP}),  let us introduce some definitions which will extend the ones used in the one-dimensional case.
\begin{definition}\label{soldefMix}
Let $g \in B[\textbf{a}, \textbf{b}]$,  and $\lambda \ge  0$. A function $w \in C_{\textbf{a}}[\textbf{a}, \textbf{b}]$   is said to solve the RL type problem  $(-\,\textbf{D}_{\textbf{a}+}^{(\nu)}, \lambda,g, 0)$ as 
\begin{itemize}
 \item [(i)]  a \emph{solution in the domain of the generator}  if $w$ satisfies (\ref{mixedP}) and belongs to  $\hat{\mathfrak{D}}_{\textbf{a}+}^{(\nu)}$; \item [(ii)] a \emph{generalized solution}  if for all sequence of functions $g_n \in C_{\textbf{a}}[\textbf{a}, \textbf{b}]$ such that $\sup_n ||g_n|| < \infty$ uniformly on $n$, and $ g_n \to g$ a.e.,  it holds that  $w(\textbf{t}) = \lim_{n\to \infty} w_n(\textbf{t})$ for all $\textbf{t} \in [\textbf{a},\textbf{b}]$, where $w_n$ is the solution (in the domain of the generator)  to the RL type problem $(-\,\textbf{D}_{\textbf{a}+}^{(\nu)}, \lambda,g_n, 0)$.
 \end{itemize}
 \end{definition}

 \begin{remark} By definition,  if there exists a generalized solution, then this is unique.
\end{remark}
For the sake of transparency, hereafter we  restrict ourselves to the analysis  for  $d=2$ and  $\textbf{a}= \textbf{0}$. 
Namely, let $\textbf{t}=(t_1, t_2),  \textbf{b}=(b_1, b_2) \in \mathbb{R}^2$, $\textbf{t} \in [  \textbf{0},\textbf{b} ]$. Let us consider the equation 
\begin{equation} \nonumber \begin{array}{rcll}
 -  ~_{t_1}\!D_{0+}^{(\nu_1)}  w(t_1,t_2)  -  ~_{t_2}\!D_{0+}^{(\nu_2)}  w(t_1,t_2) &=& \lambda u(t_1,t_2) - g(t_1,t_2),&   \\
 w (0,t_2)= w(t_1,0) &=& 0, & \,\, 
\end{array}
\end{equation}
where $t_i \in (0,b_i]$ for $i=1,2$. 

Let   $p_s^{+(\nu_i)}(t_i,r)$  (resp.   $p_s^{0+(\nu_i)}(t_i,r)$)  denote the transition density function of  the process $T^{+(\nu_i)}_{t_i}$  (resp. $T^{0+(\nu_i)}_{t_i}$) during the interval $[0,s]$.   If $\tau_{0}^{t_i,(\nu_i)} $ is the  first exit time from $(0, b_i]$ of the   process $T^{+(\nu_i)}_{t_i}$ (started at $t_i$), then the first exit time from  $(\textbf{0},\textbf{b}]=(0,b_1] \times (0,b_2]$ of the process $\textbf{T}_{\textbf{t}}^{+(\nu)}$, denoted by $\tau_{\textbf{0}}^{\textbf{t},(\nu)}$,  equals
\begin{equation}\nonumber %\label{minTd}
\tau_{\textbf{0}}^{\textbf{t},(\nu)} = \min \left \{  \,\, \tau_{0}^{t_i,\nu_i}  \,\,:\,\,   i=1,2 \right \}.
\end{equation}
Due to the independence between the coordinates  of the process $\textbf{T}_{\textbf{t}}^{+(\nu)}$,   its transition density function, denoted by  $\textbf{p}_s^{+(\nu)}(\textbf{t},\textbf{r})$,  satisfies 
\begin{equation}\nonumber %label{independence2d}
\textbf{p}_s^{+(\nu)}(\textbf{t},\textbf{r}) = \prod_{i=1}^2p_s^{+(\nu_i)}(t_i,r_i),
\end{equation}
yielding the following result.  
\begin{lemma}\label{densities2d} Let $\textbf{t}= (t_1,t_2) \in (0,b_1] \times (0,b_2]$. Suppose (H0)-(H1) hold for both functions  $\nu_1$ and $\nu_2$. Then,
\begin{itemize} 
\item [(i)]  The boundary points $(0,t_2)\in \mathbb{R}^2$ for all $t_2 \in [0,b_2)$, and $(t_1,0)\in \mathbb{R}^2$ for all $t_1 \in [0,b_1)$, are regular in expectation for both operators  $-\textbf{D}_{\textbf{0}+}^{(\nu)}$ and $-\textbf{D}_{\textbf{0}+*}^{(\nu)}$.  Moreover,  $\mathbf{E}\left [\tau_{\textbf{0}}^{\textbf{t},(\nu)}\right ] < + \infty $ uniformly on $\textbf{t}$.

\item [(ii)] If additionally  each $\nu_i$ satisfies  assumptions (H2)-(H3) and  $\mu_{\textbf{0}}^{\textbf{t}, (\nu)} (ds)$ denotes the probability law of $\tau_{\textbf{0}}^{\textbf{t},(\nu)}$,  then its density function  $\mu_{\textbf{0}}^{\textbf{t},(\nu)} (s)$ is given by
   \begin{equation}\nonumber %\label{lawofTauD2d}
\mu_{\textbf{0}}^{\textbf{t},(\nu)} (s)  \,=\,   \mu_{0}^{t_1,(\nu_1)}(s) \int_{0}^{t_2}  p_s^{+(\nu_2)} (t_2,r) + \mu_{0}^{t_2,(\nu_2)}(s) \int_{0}^{t_1}  p_s^{+(\nu_1)} (t_1,r), \quad s \ge 0.
    \end{equation} 
 
\item [(iii)] Further, assuming again that  each $\nu_i$ also satisfies   (H2)-(H3), the joint distribution of the pair $(\textbf{T}_{\textbf{t}}^{\textbf{0}+(\nu)} (s), \tau_{\textbf{0}}^{\textbf{t},(\nu)})$, denoted by $\varphi_{s,\textbf{a}}^{\textbf{t},(\nu)} (d\textbf{r}, d\xi)$, has the density
    \begin{align*}
    \varphi_{s,\textbf{0}}^{\textbf{t},(\nu)} (\textbf{r},\xi)  \,=\,\,& \varphi_{s,0}^{t_2,(\nu_2)} (r_2, \xi) p_s^{+(\nu_1)} (t_1,r_1) \int_{0}^{r_1} p_{\xi-s}^{+(\nu_1)} (r_1,y)dy  \, \,+ \\
   &+  \varphi_{s,0}^{t_1,(\nu_1)} (r_1, \xi) p_s^{+(\nu_2)} (t_2,r_2) \int_{0}^{r_2} p_{\xi-s}^{+(\nu_2)} (r_2,y)dy,    
    \end{align*}
 for all $ 0 \le s < \xi$ and $\textbf{r} = (r_1,r_2) \in (\textbf{0}, \textbf{t}]$.   
\end{itemize}
\end{lemma}
\begin{proof} $(i)$ The regularity in expectation of  the boundary $\partial_{\textbf{0}} (\textbf{0},\textbf{b}]$  is a consequence of   assumption (H1) and the Lyapunov method  (same reasoning as in \cite{KVFDE}) applied to  the  Lyapunov function 
\[
 h_{\omega}(t_1, t_2)  =  t_1 ^{\omega_1} t_2^{\omega_2},  \quad \omega_1, \omega_2 \in (0,1).
 \] %for some constants $$.
The finite expectation of $\tau_{\textbf{0}}^{\textbf{t},(\nu)}$ is a consequence of  the finite expectation of each $\tau_0^{t_i, (\nu_i)}$.

\vspace{0.2cm}
\noindent (ii)  This is a generalization of Proposition \ref{muTau} and follows  directly by differentiating
 \begin{align}
  \mathbf{P}\left [\tau_{\textbf{0}}^{\textbf{t},(\nu)}  > s\right ]  &=   \mathbf{P}\left [\tau_{0}^{t_1,(\nu_1)}  > s\right ]  \mathbf{P}\left [\tau_{0}^{t_2,(\nu_2)}  > s\right ], 
    \nonumber % \label{equivalence02d} 
   \end{align}
    with respect to $s$. Notice the use of the independence assumption in the previous equality.

\vspace{0.2cm}
\noindent (iii) This is a generalization of Proposition \ref{jointP} and is obtained  by differentiating   
\begin{align*}
\mathbf{P} \left[Ê \mathbf{T}_{\mathbf{t}}^{\mathbf{0}+(\nu)} (s)> \mathbf{r}, \tau_{\mathbf{0}}^{\mathbf{t},(\nu)} > \xi \right ] 
&= \prod_{i=1}^2 \mathbf{P} \left[Ê T_{t_i}^{0+(\nu_i)} (s) > r_i,T_{t_i}^{0+(\nu_i)} (\xi) > 0\right ],
\end{align*}
with respect to $r_1, r_2$ and $\xi$. %The previous equality holds due to the independence assumption between the two processes.
\end{proof}

Let us now  generalize the definitions given in (\ref{defM}) and (\ref{defM2}).
For $\lambda \ge 0$  and  $g\in B[\textbf{0},\textbf{b}]$  define   
\begin{equation}\nonumber \label{defM2d}
 \textbf{M}_{\textbf{0},\lambda}^{+(\nu)} g(\textbf{t})  := \mathbf{E} \left [  \int_0^{\tau_{\textbf{0}}^{\textbf{t}, (\nu)}} e^{-\lambda s} g (\textbf{T}_{\textbf{t}}^{+(\nu)} (s)) ds\right ],\quad \textbf{t} \in (\textbf{0},\textbf{b}],  
 \end{equation}
 and 
\begin{equation} \nonumber %\label{defM2d}
 \textbf{M}_{\textbf{0},\lambda}^{+(\nu)} 1(\textbf{t})  := \mathbf{E} \left [  \int_0^{\tau_{\textbf{0}}^{\textbf{t}, (\nu)}} e^{-\lambda s}  ds\right ],\quad \textbf{t} \in [\textbf{0},\textbf{b}].  
 \end{equation}
 Note that $\textbf{M}_{\textbf{0},\lambda}^{+(\nu)} g(\cdot)$ is continuous on $[\textbf{0},\textbf{b}]$ and $ | \textbf{M}_{\textbf{0},\lambda}^{+(\nu)}  g(\textbf{t})| \le ||g|| \sup_{\textbf{t} \in [\textbf{0},\textbf{b}]} \mathbf{E} \left[ \tau_{\textbf{0}}^{\textbf{t},(\nu)}\right ]$. Moreover,  
\begin{equation}\nonumber %\label{M12d}
   \textbf{M}_{\textbf{0},\lambda}^{+(\nu)} \cdot 1 (\textbf{t})=  \frac{1}{\lambda} \left ( 1 - \mathbf{E} \left[ e^{-\lambda \tau_{\textbf{0}}^{\textbf{t}, (\nu)}}\right ]\right ), 
\end{equation}
implying
\begin{equation}\nonumber %\label{laplaceM2d}
  \mathbf{E} \left[ e^{-\lambda \tau_{\textbf{0}}^{\textbf{t}, (\nu)}}\right ] =  1 - \lambda \textbf{M}_{\textbf{0},\lambda}^{+(\nu)} \cdot 1 (\textbf{t}), 
\end{equation}
and yielding the next  generalization of Lemma \ref{MLaplace-MJoint}.
  \begin{lemma}\label{MLaplace-MJoint2d}
Let  $\textbf{t}=(t_1,t_2) \in (\textbf{0},\textbf{b}]$ and $\lambda > 0$. Suppose that  $\nu_i$ satisfies conditions (H0)-(H1),  for   $i=1,2$. Then,
\begin{align}\nonumber
\mathbf{E} \left[ e^{-\lambda \tau_{\textbf{0}}^{\textbf{t}, (\nu)}}\right ]  &=  \mathbf{E} \left[ e^{-\lambda \tau_{0}^{t_1, (\nu_1)}} 1_{\left \{ \tau_{0}^{t_1, (\nu_1)} < \tau_{0}^{t_2, (\nu_2)} \right \}}\right ] +\mathbf{E} \left[ e^{-\lambda \tau_{0}^{t_2, (\nu_2)}} 1_{\left \{ \tau_{0}^{t_2, (\nu_2)} < \tau_{0}^{t_1, (\nu_1)} \right  \}}\right ]. 
\end{align}
If additionally $\nu_i$ satisfies (H2)-(H3) for  each $i=1,2$,  then 
\begin{align*}
\mathbf{E} \left[ e^{-\lambda \tau_{\textbf{0}}^{\textbf{t}, (\nu)}}\right ] = &\int_0^{\infty} e^{-\lambda s} \left(    \mu_{0}^{t_1,(\nu_1)}(s) \int_{0}^{t_2}  p_s^{+(\nu_2)} (t_2,r) \right) ds \,\,+\\
& + \int_0^{\infty} e^{-\lambda s} \left(  \mu_{0}^{t_2,(\nu_2)}(s) \int_{0}^{t_1}  p_s^{+(\nu_1)} (t_1,r) \right)ds.
\end{align*}
Further, 
\begin{equation}\label{explicitD} %\label{MJoint2d}
\textbf{M}_{\textbf{0},\lambda}^{+(\nu)} g(\textbf{t})  = \int_0^{t_1}\int_0^{t_2} g(t_1-r_1,\,t_2-r_2 ) \int_0^{\infty} e^{-\lambda s}    p_s^{+(\nu_1)}(t_1,t_1-r_1)p_s^{+(\nu_2)}(t_2,t_2-r_2)  \, ds\, dr_2\,dr_1.
 \end{equation}
\end{lemma}
\begin{proof}
Similar to the proof of Lemma \ref{MLaplace-MJoint} but using the density function of the r.v. $\tau_{\textbf{0}}^{\textbf{t}, (\nu)}$ and the joint distribution of the pair $(\textbf{T}_{\textbf{t}}^{\textbf{0}+(\nu)} (s), \tau_{\textbf{0}}^{\textbf{t},(\nu)})$   given in Lemma \ref{densities2d}.
\end{proof}
%%%%

\noindent \textbf{Well-posedness result for the RL type linear equation.}
\begin{theorem} (\textbf{Case $\lambda > 0$})\label{wellposedness12d}
Let $\nu=(\nu_1,\nu_2)$ be a vector such that each $\nu_i$ is a function satisfying conditions (H0)-(H1). Suppose $\lambda > 0$ and $(t_1,t_2) \in (0,b_1]\times (0,b_2]$.  
\begin{itemize}
\item [(i)] If $g\in C_{\textbf{0}}[\textbf{0}, \textbf{b}]$, then the  equation $(-  ~_{t_1}\!D_{0+}^{(\nu_1)}-  ~_{t_2}\!D_{0+}^{(\nu_2)} , \lambda, g,0)$ has a unique solution in the domain of the generator  given by $w = \textbf{R}_{\lambda}^{\textbf{0}+(\nu)} g$, the resolvent operator  of the process $\textbf{T}_{\textbf{t}}^{\textbf{0}+(\nu)}$.
\item [(ii)] For any $g \in B[\textbf{0},\textbf{b}]$, the mixed equation $(-  ~_{t_1}\!D_{0+}^{(\nu_1)}-  ~_{t_2}\!D_{0+}^{(\nu_2)} , \lambda, g,0)$ is well-posed in the generalized sense and the solution  admits the stochastic representation
\begin{equation}\label{MsolRL2d}
w(t_1,t_2) =  \mathbf{E} \left[Ê \int_0^{\mathbf{\tau}_{ \textbf{0}}^{(t_1,t_2),(\nu)}}  e^{-\lambda s} g \left(T_{t_1}^{0+(\nu_1)} (s)\,\,,\,\,T_{t_2}^{0+(\nu_2)} (s)  \right ) ds\right ].
\end{equation} 
Moreover, if additionally each $\nu_i$, $i=1,2$,  satisfies conditions (H2)-(H3), then   $w(t_1,t_2)$ takes the explicit form in (\ref{explicitD}).% \begin{equation}\label{solTheoremRL2d}
%w(t_1,t_2) =  \int_0^{t_1} \int_0^{t_2} g(t_1-r_1, t_2-r_2)  \int_0^{\infty} e^{-\lambda s} p_s^{+(\nu_1)}(t_1,t_1-r_1) p_s^{+(\nu_2)}(t_2,t_2-r_2)  \,ds\,\, dr_2 \,\,dr_1. 
%\end{equation}
% \item [(iii)] If $g \in C^1[\textbf{0},\textbf{b}]$,  the assumptions (H2)-(H3)  also hold, and $\textbf{p}_s^{+(\nu)}(\textbf{t},\textbf{r})$ is continuously differentiable in both variables $\textbf{t}$ and  $\textbf{r}$,   then  $w \in C_{\textbf{0}}[\textbf{0},\textbf{b}] \cap C^1(\textbf{0},\textbf{b}]$. Moreover, if $g \in C_{\textbf{0}}^1 [\textbf{0},\textbf{b}]$, then $w \in C^1[\textbf{0},\textbf{b}]$.
 \end{itemize}
 \end{theorem}
\begin{proof}
(i)  Follows from Theorem 1.1 in  \cite{dynkin1965} as in the one-dimensional case.%  as $-\,\textbf{D}_{\textbf{0}+}^{(\nu)}$ generates a Feller process on $C_{\textbf{0}}[\textbf{0},\textbf{b}]$,  if $g \in C_{\textbf{0}}[\textbf{0},\textbf{b}]$,  then the resolvent operator of the semigroup of $\textbf{T}_{\textbf{t}}^{\textbf{0}+(\nu)}$ provides  the unique solution (in the domain of the generator) to the  linear problem  for $(-\,\textbf{D}_{\textbf{0}+}^{(\nu)}, \lambda, g,0)$.  

\vspace{0.1cm}
\noindent (ii) If $g \in B[\textbf{0},\textbf{b}]$, the solution is obtained as a limit of solutions $\textbf{R}_{\lambda}^{\textbf{0}+(\nu)} g_n(\textbf{t})$ in the domain of the generator $-\mathbf{D}_{\textbf{0}+}^{(\nu)}$, where the sequence of functions $\{g_n\}_{n\ge 1}$ satisfies the  conditions  of Definition \ref{soldefMix}.  
 Finally, Lemma \ref{MLaplace-MJoint2d} provides the explicit representation of the solution $w$ in terms of transition densities.
 \end{proof}

\begin{theorem} (\textbf{Case $\lambda = 0$})
All assertions in Theorem \ref{wellposedness12d} are valid for $\lambda = 0$.
\end{theorem}
\begin{proof}
The arguments in the proof of Theorem \ref{wellposedness12d} remain valid for the case $\lambda = 0$ replacing the resolvent $\mathbf{R}_{\lambda}^{\textbf{0}+(\nu)}$ by the  corresponding potential  operator $\textbf{R}_{0}^{\textbf{0}+(\nu)}$.
\end{proof}
Finally, we shall analyze  the  mixed linear equation which involves both the RL type and the Caputo type operator:
\begin{equation}\label{mixedRLC}
\begin{array}{rll}
 -  ~_{t_1}\!D_{0+}^{(\nu_1)}  u(t_1,t_2)  -  ~_{t_2}\!D_{0+*}^{(\nu_2)}  u(t_1,t_2)
 =& \lambda u(t_1,t_2) - g(t_1,t_2),& (t_1, t_2) \in (0,b_1]\times(0,b_2],  \\
 u (0,t_2)=& 0, &\, t_2 \in [0,b_2] \\
  u(t_1,0) =& \phi(t_1) &  t_1\in (0,b_1],
\end{array}
\end{equation}
for a given function $\phi \in C_{0}[0,b_1]$. This equation will be  referred  to  as the \emph{mixed linear problem} $(-  ~_{t_1}\!D_{0+}^{(\nu_1)} -  ~_{t_2}\!D_{0+*}^{(\nu_2)}, \lambda,g, \phi)$.

Denote by $\textbf{T}_{\textbf{t}}^{\textbf{0}+(\nu)*}:= \left (T_{t_1}^{0+(\nu_1)}, T_{t_2}^{0+*(\nu_2)} \right)$ the Feller process (with values on  $(0,b_1] \times [0,b_2]$) generated by the operator   $-  ~_{t_1}\!D_{0+}^{(\nu_1)} -  ~_{t_2}\!D_{0+*}^{(\nu_2)}$.  This process  is obtained from a process $\textbf{T}_{\textbf{t}}^{+(\nu)}:=\left (T_{t_1}^{+(\nu_1)}\,\,,\,\,T_{t_2}^{+(\nu_2)} \right )$ by either killing it  whether the first coordinate attempt to cross the boundary point $t_1 = 0$,  or by stopping it if the second coordinate does the same with the boundary point $t_2=0$. As before,  $\tau_{\textbf{0}}^{\textbf{t},(\nu)}$ denotes the first exit time from $(0,b_1]\times(0,b_2]$.

In order to solve the mixed  equation (\ref{mixedRLC}), we shall rewrite it   as a linear equation involving only RL type operators.  Namely,  let $\psi \in C([0,t_1]\times[0,t_2])$  be a function satisfying the boundary conditions in (\ref{mixedRLC}). Define $w(\textbf{t}):= u(\textbf{t}) - \psi(\textbf{t})$ for any $\textbf{t}=(t_1, t_2) \in [\textbf{0}, \textbf{b}]$. Observe that, by definition, $w$ vanishes at the boundary $\partial_{\textbf{0}}[\textbf{0}, \textbf{b}]$.  

 If  $u$ and $\psi$ belong to the domain of the generator $-\,~_{t_1}\!D_{0+}^{(\nu_1)} -  ~_{t_2}\!D_{0+*}^{(\nu_2)}$, then
\begin{align*}
\left( -\,~_{t_1}\!D_{0+}^{(\nu_1)} -  ~_{t_2}\!D_{0+*}^{(\nu_2)} \right ) \, w &=  \left( -\,~_{t_1}\!D_{0+}^{(\nu_1)} -  ~_{t_2}\!D_{0+*}^{(\nu_2)} \right ) \, u + \left( \,~_{t_1}\!D_{0+}^{(\nu_1)} +  ~_{t_2}\!D_{0+*}^{(\nu_2)} \right ) \, \psi \\
&=  \lambda u - \left [\, g \,- \,\left( \,~_{t_1}\!D_{0+}^{(\nu_1)} +  ~_{t_2}\!D_{0+*}^{(\nu_2)} \right ) \, \psi  \, \right ] =: \lambda w - \tilde{g},
\end{align*}
with $\tilde{g} : = \,g  - \lambda \psi -  ~_{t_1}\!D_{0+}^{(\nu_1)} \psi-  ~_{t_2}\!D_{0+*}^{(\nu_2)}\psi$.
Due to the properties satisfied by $\psi$, the function $w$ satisfies that $-    ~_{t_2}\!D_{0+*}^{(\nu_2)}w =  -  ~_{t_2}\!D_{0+}^{(\nu_2)} w=0$ at the boundary $\partial_{\textbf{0}}(\textbf{0}, \textbf{b}]$.   Consequently, the solution $u$ to  (\ref{mixedRLC}) can be written as $u=w + \psi$, where $w$ is the solution to the corresponding  RL type equation. This leads us to the next definition.

 \begin{definition} \label{solDefMixed}
Let $g \in B[\textbf{0}, \textbf{b}]$, $\lambda \ge 0$, and $\phi \in C_0[0,b_1]$.  A function  $u \in C[\textbf{0}, \textbf{b}]$ is  said to solve the mixed linear problem   $(-  ~_{t_1}\!D_{0+}^{(\nu_1)} -  ~_{t_2}\!D_{0+*}^{(\nu_2)}, \lambda,g, \phi)$ as
\begin{itemize}
\item [(i)] a \emph{solution in the domain of the generator}  if $u$ satisfies (\ref{mixedRLC}) and belongs to the domain of the generator $-  ~_{t_1}\!D_{0+}^{(\nu_1)} -  ~_{t_2}\!D_{0+*}^{(\nu_2)}$; \item [(ii)] a \emph{generalized solution}  if for any function $\psi$ in the domain of $-  ~_{t_1}\!D_{0+}^{(\nu_1)} -  ~_{t_2}\!D_{0+*}^{(\nu_2)}$ such that $\psi (0,\cdot) = 0$ and $\psi(\cdot, 0) = \phi (\cdot)$, then $u = \omega + \psi$,  where $\omega$ is  a solution (possibly generalized) to the RL type problem  \[ (-  ~_{t_1}\!D_{0+}^{(\nu_1)} -  ~_{t_2}\!D_{0+}^{(\nu_2)}, \lambda,\tilde{g}, 0), \]
with $\tilde{g} : = \,g  - \lambda \psi -  ~_{t_1}\!D_{0+}^{(\nu_1)} \psi-  ~_{t_2}\!D_{0+*}^{(\nu_2)}\psi$. 
 \end{itemize}
\end{definition}
 \begin{remark} By definition, it seems that a generalized solution depends on the function $\psi$, the next result we will show  that this is actually independent of $\psi$.
 \end{remark}
\begin{remark}
The concept of a generalized solution can  also  be given as a limit of  solutions in the domain of the generator. 
\end{remark}
\noindent \textbf{Well-posedness result for the mixed  linear equation.}
\begin{theorem}(\textbf{Case $\lambda > 0$}) \label{wellposednessMixed}
Let $\nu=(\nu_1,\nu_2)$ be a vector such that each $\nu_i$ is a function satisfying conditions (H0)-(H1). Suppose $\lambda > 0$ and $\phi \in C_{0}[0,b_1]$.   
\begin{itemize}
\item [(i)]  If $g\in C[\textbf{0},\textbf{b}]$ satisfies $g(0, \cdot) \equiv 0$ and $g(\cdot,0) = \lambda \phi (\cdot)$,  then the mixed equation $(-  ~_{t_1}\!D_{0+}^{(\nu_1)}-  ~_{t_2}\!D_{0+*}^{(\nu_2)} , \lambda, g,\phi)$ has a unique solution in the domain of the generator   given by  $u = \mathfrak{R}_{\lambda}^{\textbf{0}+(\nu)*} g$, the resolvent operator of the process $ \left( \,\,T_{t_1}^{0+(\nu_1)}, \,\,T_{t_2}^{0+*(\nu_2)} \,\, \right )$.
\item [(ii)] For any $g \in B[\textbf{0}, \textbf{b}]$, the mixed linear equation   $(-  ~_{t_1}\!D_{0+}^{(\nu_1)} -  ~_{t_2}\!D_{0+*}^{(\nu_2)}, \lambda, g,\phi)$ is well-posed in the generalized sense and the solution admits the stochastic representation   
\begin{align}\nonumber
u(t_1,t_2) &=  \mathbf{E} \left[ e^{-\lambda  \tau_0^{t_2,(\nu_2)}} Ê  \phi \left( T_{t_1}^{0+(\nu_1)} \left ( \tau_{0}^{t_2,(\nu_2)} \right )  \right )\,\, 1_{\left \{ \tau_{0}^{t_2,(\nu_2)} < \tau_{0}^{t_1,(\nu_1)} \right\}}\right ]  \\ &+ \mathbf{E} \left [  \int_0^{\tau_{\textbf{0}}^{\textbf{t},(\nu)}}  e^{-\lambda s}  g \left( T_{t_1}^{0+(\nu_1)} (s), \,\,T_{t_2}^{0+*(\nu_2)} (s) \right )\right ]. \label{SRMixed}
\end{align}
Moreover, if additionally each $\nu_i$, $i=1,2$,  satisfies conditions (H2)-(H3), then   the solution rewrites as
\begin{align}
u &(t_1,t_2) \,\,= \,\,\int_0^{t_1} \phi (t_1 -r) \int_0^{\infty} e^{-\lambda s} \mu_{0}^{t_2, (\nu_2)} (s) p_s^{+(\nu_1)} (t_1,t_1 - r_1) \,\,ds\,\, dr \, +  \nonumber \\  \label{explicitMixed}
 &+  \int_0^{t_1} \int_0^{t_2} g(t_1-r_1, t_2-r_2)  \int_0^{\infty} e^{-\lambda s} p_s^{+(\nu_1)}(t_1,t_1-r_1) p_s^{+(\nu_2)}(t_2,t_2-r_2)  \,ds\,\, dr_2 \,\,dr_1. 
\end{align}
 %\item [(iii)] If $g \in C^1[\textbf{0},\textbf{b}]$, $\phi \in C^1[0,b_1]$,  the  assumptions (H2)-(H3) also hold, and $p_s^{+(\nu_1)}(t,r)$ and $p_s^{+(\nu_2)}(t,r)$ are  continuously differentiable in both variables $t$ and $r$,   then  $u \in C [\textbf{0},\textbf{b}] \cap C^1(\textbf{0},\textbf{b}]$. %Moreover, if $g(0, \cdot) \equiv 0$ and $g(\, \cdot \,, 0) = \lambda \phi (\cdot)$, then $u \in C^1[\textbf{0},\textbf{b}]$.
 \end{itemize}
 \end{theorem}
\begin{proof} (i) As before, we apply Theorem 1.1 in \cite{dynkin1965} to the generator $-  ~_{t_1}\!D_{0+}^{(\nu_1)} -  ~_{t_2}\!D_{0+*}^{(\nu_2)}$. Therefore, if $g$ is a  continuous function on $[0,b_1] \times [0,b_2]$ such that  $g(0,\,\cdot \,) \equiv 0$, then  the function $u(t_1,t_2) = \mathfrak{R}_{\lambda}^{\textbf{0}+(\nu)*} g(t_1,t_2)$ solves the equation (\ref{mixedRLC}) without any boundary condition. Further, a simple calculation shows that \[u(t_1,0) = \mathfrak{R}_{\lambda}^{\textbf{0}+(\nu)*} g(t_1,0)= g(t_1, 0)/\lambda,\]  which implies that, under condition $g(\,\cdot \,, 0) = \lambda \phi (\cdot)$, the function $u$  solves the problem (\ref{mixedRLC}). 

%TESIS
%Further, by  assumption (H1)  the first exit time from $(0,t_1]\times (0,t_2]$ is finite a.s. implying
%\begin{align}\nonumber
%\mathfrak{R}_{\lambda}^{\textbf{0}+(\nu)*} g(t_1,t_2) &=  \frac{1}{\lambda}\mathbf{E} \left[ e^{-\lambda \tau_{0}^{t_2,(\nu_2)}} Ê   g \left( T_{t_1}^{0+(\nu_1)} \left ( \tau_{0}^{t_2,(\nu_2)} \right )\,\, ,\,\, 0  \right )  \,\, 1_{\left \{ \tau_{0}^{t_2,(\nu_2)} < \tau_{0}^{t_1,(\nu_1)} \right\}} \right ]  \\ &+ \mathbf{E} \left [  \int_0^{\tau_{\textbf{0}}^{\textbf{t},(\nu)}}  e^{-\lambda s}  g \left( T_{t_1}^{0+(\nu_1)} (s), \,\,T_{t_2}^{0+*(\nu_2)} (s) \right )\right ]. \nonumber %\label{SRMixed----new}
%\end{align}

\vspace{0.2cm}
\noindent (ii) For the general case  $g\in B[\textbf{0},\textbf{b}]$, take a function $\psi$ satisfying the conditions of  Definition  \ref{solDefMixed} and set $w:= u - \psi$. Since $w$ vanishes at the boundary  $\partial_{\textbf{0}}(\textbf{0},\textbf{b}]$,   Theorem \ref{wellposedness12d}  
 yields   \begin{equation}\nonumber
w(t)  =  \mathbf{E} \left[Ê \int_0^{\mathbf{\tau}_{ \textbf{0}}^{(t_1,t_2),(\nu)}}  e^{-\lambda s} \tilde{g} \left(T_{t_1}^{0+(\nu_1)} (s)\,\,,\,\,T_{t_2}^{0+(\nu_2)} (s)  \right ) ds\right ],
\end{equation}
with $\tilde{g} = g -  \,\lambda \psi - (   ~_{t_1}\!D_{0+}^{(\nu_1)} +  ~_{t_2}\!D_{0+*}^{(\nu_2)}) \, \psi$. 
%Further, by  assumption (H1)  the first exit time from $(0,t_1]\times (0,t_2]$ is finite a.s. implying
Hence $w (t)  =  \,I - \, II $, where 
\begin{align*}
 \,I -& \, II :=\mathbf{E} \left[Ê \int_0^{\mathbf{\tau}_{ \textbf{0}}^{(t_1,t_2),(\nu)}}  e^{-\lambda s} g \left(T_{t_1}^{0+(\nu_1)} (s)\,\,,\,\,T_{t_2}^{0+(\nu_2)} (s)  \right ) ds\right ]  \\
&-  \mathbf{E} \left[Ê \int_0^{\mathbf{\tau}_{ \textbf{0}}^{(t_1,t_2),(\nu)}}  e^{-\lambda s}  (\lambda  +~_{t_1}\!D_{0+}^{(\nu_1)}  
+   ~_{t_2}\!D_{0+*}^{(\nu_2)}) \, \psi \left(T_{t_1}^{0+(\nu_1)} (s)\,\,,\,\,T_{t_2}^{0+(\nu_2)} (s)  \right ) ds\right ].
\end{align*}
Using that $\psi$ belongs to the domain of the generator $ -~_{t_1}\!D_{0+}^{(\nu_1)}  
-   ~_{t_2}\!D_{0+*}^{(\nu_2)}$,   Doob's stopping theorem applied to the martingale

\begin{align}\label{MtgleMixed}
e^{-\lambda r} &\psi \left(T_{t_1}^{0+(\nu_1)} (r)\,\,,\,\,T_{t_2}^{0+*(\nu_2)} (r)  \right ) + \\ &+ \int_0^r e^{-\lambda s} \left( \lambda  +~_{t_1}\!D_{0}^{(\nu_1)}  
+  ~_{t_2}\!D_{0+*}^{(\nu_2)} \right ) \psi \left(T_{t_1}^{0+(\nu_1)} (s)\,\,,\,\,T_{t_2}^{0+*(\nu_2)} (s)  \right )ds  \nonumber \end{align}
%operator $ -~_{t_1}\!D_{0}^{(\nu_1)}  -   ~_{t_2}\!D_{0+*}^{(\nu_2)}$ 
with the stopping time $\mathbf{\tau}_{ \textbf{0}}^{(t_1,t_2),(\nu)}$ implies
\begin{equation}\nonumber
 II\,= \,  \psi (t_1, t_2)\,-\, \mathbf{E} \left [ e^{-\lambda \tau_{\textbf{0}}^{(t_1,t_2),(\nu)} } \psi \left( T_{t_1}^{0+(\nu_1)} (\tau_{\textbf{0}}^{(t_1,t_2),(\nu)})\,\,,\,\,T_{t_2}^{0+(\nu_2)} (\tau_{\textbf{0}}^{(t_1,t_2),(\nu)})  \right )  \right ],  
\end{equation}
which, in turn, yields   (\ref{SRMixed}) as  $u = w + \psi$ and $\psi (\cdot,0) = \phi (\cdot)$.

Finally, the second term in (\ref{explicitMixed}) is a consequence of Lemma \ref{MLaplace-MJoint2d}, whilst  the first term is obtained by conditioning first on $\tau_{0}^{t_2,(\nu_2)}$ and then by using the joint density of the pair $(T_{t_1}^{0+(\nu)}(s), \tau_0^{t_1,(\nu_1)})$. %Finally,  the differentiability of the explicit solution (\ref{explicitMixed}) (under the assumptions stated in $(iii)$)  implies the last claim. 
\end{proof}
\begin{theorem} \textbf{(Case $\lambda= 0$) } 
All the assertions in Theorem \ref{wellposednessMixed} are valid for the case $\lambda = 0$.
\end{theorem}
\begin{proof}
For functions $g\in C_{\textbf{0}}[\textbf{0},\textbf{b}]$, the arguments in the proof of Theorem \ref{wellposednessMixed} remain valid using the potential operator  $\mathfrak{R}_{0}^{\textbf{0}+(\nu)*}$ instead of the   resolvent operator  $\mathfrak{R}_{\lambda}^{\textbf{0}+(\nu)*}$.  For the general case $g \in B[\textbf{0}, \textbf{b}]$,  the  martingale (\ref{MtgleMixed}) should be replaced by the corresponding martingale with $\lambda= 0$.
%\begin{equation}\nonumber
 %\psi \left(T_{t_1}^{0+(\nu_1)} (r)\,,\,T_{t_2}^{0+*(\nu_2)} (r)  \right ) - \int_0^r \left ( ~_{t_1}\!D_{0}^{(\nu_1)}  
%+  ~_{t_2}\!D_{0+*}^{(\nu_2)} \right ) \psi \left(T_{t_1}^{0+(\nu_1)} (s)\,\,,\,\,T_{t_2}^{0+*(\nu_2)} (s)  \right )ds.  \end{equation}
\end{proof}

\begin{remark}
As an application of  Theorem \ref{wellposednessMixed}, one  obtains that  for $\textbf{t}=(t_1,t_2)$, the function 
\small
\begin{align*}
&u(\textbf{t}) =
\frac{1}{\alpha} t_2 \int_0^{t_1} \phi (t_1 -r) \int_0^{\infty} e^{-\lambda s}\,\,s^{-\frac{1}{\alpha} -\frac{1}{\beta}- 1} w_{\alpha} \left( t_2 s^{-1/\alpha};1,1 \right ) w_{\beta} \left( r_1 s^{-1/\beta};1,1 \right )  \,\,ds\,\, dr \, +  \\%\label{explicitMixed-Ex}
 &+  \int_0^{t_1} \int_0^{t_2} g(t_1-r_1, t_2-r_2)  \int_0^{\infty} e^{-\lambda s} s^{-\frac{1}{\beta}-\frac{1}{\alpha}} w_{\beta} \left( r_1 s^{-1/\beta};1,1 \right )w_{\alpha} \left( r_2 s^{-1/\alpha};1,1 \right )ds\, dr_2\,dr_1 
\end{align*}
\normalsize
 is the generalized solution to the  \textit{ mixed fractional linear  equation}  
\begin{equation}\nonumber %\label{mixedRLC}
\begin{array}{rcll}
 -  ~_{t_1}\!D_{0+}^{\beta}  u(t_1,t_2)  -  ~_{t_2}\!D_{0+*}^{\alpha}  u(t_1,t_2)
 &=& \lambda u(t_1,t_2) - g(t_1,t_2),& (t_1,t_2) \in (\textbf{0},\textbf{b}],  \\
 u (0,t_2)&=& 0, &\, t_2 \in [0,b_2] \\
  u(t_1,0) &=& \phi(t_1) &  t_1\in (0,b_1],
\end{array}
\end{equation}
for a given function $\phi \in C_{0}[0,b_1]$ and $\beta, \alpha \in (0,1)$.  Let us recall that $-  ~_{t_1}\!D_{0+}^{\beta} $ and $-  ~_{t_2}\!D_{0+*}^{\alpha}$ stand for the classical RL and Caputo derivatives of order $\beta$ and $\alpha$, respectively; and $w_{\beta}$ and $w_{\alpha}$ denote $\beta-$ and $\alpha-$stable densities, respectively (see Preliminaries).

The solution $u$ belongs to the domain of the generator only when $g\in C[\textbf{0},\textbf{b}]$, $g(\cdot,0) = \lambda \phi (\cdot)$ and $g(0,\cdot) \equiv 0$. 
\end{remark}
\section{Acknowledgements}
The first author is supported by Chancellor International Scholarship through the University of Warwick, UK.

\end{document}